\documentclass[11 pt]{article}
\usepackage{amsmath,amssymb,amsthm,amsfonts,verbatim}
\usepackage{enumerate}
\usepackage[all,2cell]{xy}
\usepackage{microtype}
\usepackage[vcentermath]{youngtab}
\usepackage[top=1.2in,bottom=1.2in,left=1.2in,right=1.2in]{geometry}

\title{Homological stability for configuration spaces of manifolds}
\author{Thomas Church}

\newcommand{\nc}{\newcommand}
\nc{\dmo}{\DeclareMathOperator}

\theoremstyle{plain}
\newtheorem{theorem}{Theorem}[section]
\newtheorem*{main}{Theorem 1}
\newtheorem*{thmtwo}{Corollary 2}
\newtheorem*{maincor}{Corollary 3}
\newtheorem*{cortwo}{Corollary 4}
\newtheorem*{thmvw}{Theorem 5}
\newtheorem*{propbetti}{Proposition \ref{prop:improvebetti}}
\nc\theoremmain{1}
\nc\theoremtwo{2}
\nc\corollarymain{3}
\nc\corollarytwo{4}
\nc\theoremvw{5}

\newtheorem{proposition}[theorem]{Proposition}
\newtheorem{lemma}[theorem]{Lemma}
\newtheorem*{theorem:tech}{Theorem \ref{theorem:pieri}}
\theoremstyle{definition}
\newtheorem*{definition:rs}{Definition \ref{def:repstability}}
\newtheorem{definition}[theorem]{Definition}

\newtheorem{criterion}[theorem]{Criterion}
\newtheorem{remark}[theorem]{Remark}
\newtheorem*{step1}{Step 1 (Stability of $E_2$ page)}
\newtheorem*{step2}{Step 2 (Monotonicity of $E_2$ page)}
\newtheorem*{step3}{Step 3 (Stability and monotonicity of $E_r$ page)}
\newtheorem*{step4}{Step 4 (Stability and monotonicity of $H^*(C_n(M))$)}
\newtheorem*{step5}{Step 5 (Stability of map $C_n(N)\to C_n(M)$)}

\nc{\Q}{\mathbb{Q}}
\nc{\R}{\mathbb{R}}
\nc{\Z}{\mathbb{Z}}
\nc{\C}{\mathbb{C}}
\nc{\M}{\mathcal{M}}
\nc{\V}{\mathcal{V}}
\nc{\N}{\mathbb{N}}
\nc{\cN}{\mathcal{N}}
\nc{\cS}{\mathcal{S}}
\nc\T{\mathcal{T}}
\nc\B{\mathcal{B}}
\dmo{\Hom}{Hom}
\dmo{\Ind}{Ind}
\dmo{\Stab}{Stab}
\dmo{\ColStab}{ColStab}
\dmo{\EStab}{EStab}
\dmo{\spn}{span}
\dmo{\sh}{sh}
\dmo{\im}{im}
\dmo{\Sym}{Sym}
\dmo{\GSym}{GSym}
\dmo{\supp}{supp}
\dmo{\even}{even}
\dmo{\odd}{odd}
\dmo{\Mod}{Mod}
\dmo{\Bij}{Bij}
\dmo{\coker}{coker}
\dmo{\ttop}{top}
\dmo{\GL}{GL}
\nc{\bwedge}{\textstyle{\bigwedge}}
\nc{\corr}{\longleftrightarrow}
\nc{\bi}{b_i}
\nc{\bj}{b_j}
\nc{\cell}{\,}
\nc{\capB}{B}
\nc{\capC}{C}

\renewcommand{\epsilon}{\varepsilon}
\nc{\coloneq}{\mathrel{\mathop:}\mkern-1.2mu=}

\nc{\para}[1]{\medskip\noindent\textbf{#1.}}

\begin{document}
\maketitle
\begin{abstract}
Let $C_n(M)$ be the configuration space of $n$ distinct ordered points in $M$. We prove that if $M$ is any connected orientable manifold (closed or open), the homology groups $H_i(C_n(M);\Q)$ are representation stable in the sense of [CF]. Applying this to the trivial representation, we obtain as a corollary that the unordered configuration space $B_n(M)$ satisfies classical homological stability: for each $i$, $H_i(B_n(M);\Q)\approx H_i(B_{n+1}(M);\Q)$ for $n>i$.  This improves on results of McDuff, Segal, and others for open manifolds. Applied to closed manifolds, this provides natural examples where rational homological stability holds even though integral homological stability fails.

To prove the main theorem, we introduce the notion of \emph{monotonicity} for a sequence of $S_n$--representations, which is of independent interest. 
Monotonicity provides a new mechanism for proving representation stability using spectral sequences. The key technical point in the main theorem is that certain sequences of induced representations are monotone.
\end{abstract}

\section{Introduction}
\para{Configuration spaces} For any space $X$, let \[C_n(X)\coloneq \left\{(x_1,\ldots,x_n)\in X^n\big|x_i\neq x_j\right\}\] be the configuration space of $n$ distinct ordered points in $X$. The action of the symmetric group $S_n$ on $X^n$ restricts to a free action on $C_n(X)$, and the quotient \[B_n(X)\coloneq  C_n(X)/S_n\] is the configuration space of $n$ distinct \emph{unordered} points in $X$. The first configuration spaces to be studied in depth were $C_n(\R^2)$ and $B_n(\R^2)$, aspherical spaces whose fundamental groups are the \emph{pure braid group} and the \emph{braid group} respectively. There are inclusions $B_n(\R^2)\to B_{n+1}(\R^2)$, defined for example by \[\{z_1,\ldots,z_n\}\mapsto \{z_1,\ldots,z_n,1+\sup\Re z_i\}.\] Arnol'd \cite{Ar} proved that $B_n(\R^2)$ satisfies \emph{integral homological stability} with respect to these maps, meaning that for each $i\geq 0$, the induced maps \[H_i(B_n(\R^2);\Z)\overset{\approx}{\longrightarrow}H_i(B_{n+1}(\R^2);\Z)\] are isomorphisms once $n$ is sufficiently large (depending on $i$).

\para{Previous work} There are a number of problems in generalizing this theorem to configuration spaces of other manifolds. First, Arnol'd's proof of this theorem hinges on an identification of $B_n(\R^2)$ with the space of monic squarefree degree $n$ complex polynomials, and no such identification is available in general. But the main problem is that for a general manifold $M$ there are no natural maps $B_n(M)\to B_{n+1}(M)$ to induce the desired isomorphisms on homology.

\para{Open manifolds} In the past this obstacle was avoided by restricting to open manifolds. If $M$ is the interior of a compact manifold with boundary, maps $B_n(M)\to B_{n+1}(M)$ can be defined by ``pushing the additional point in from infinity''.
Formally, fix a collar neighborhood $U$ of the boundary, and choose a homeomorphism, isotopic to the identity, between $M$ and the interior $M\setminus U$. This gives an identification $B_n(M)\simeq B_n(M\setminus U)$, and the map $B_n(M\setminus U)\to B_{n+1}(M)$ is defined by adding a fixed point lying in the collar $U$. For such manifolds $M$, McDuff \cite[Theorem 1.2]{McD} related $B_n(M)$ to a certain space of generalized vector fields to prove that $B_n(M)$ satisfies integral homological stability; taking an exhaustion by such manifolds, a direct limit argument extends the result to arbitrary open manifolds. Segal \cite[Appendix to \S5]{Se} proved the same result for open manifolds by an approach closer to Arnol'd's.

\para{Closed manifolds} However, the corresponding theorem is false for closed manifolds. The simplest example where homological stability fails is the 2--sphere, where we have (see e.g. \cite[Theorem 1.11]{Bi}): \[H_1(B_n(S^2);\Z)=\Z/(2n-2)\Z\]  Thus for all $n$ we have $H_1(B_n(S^2);\Z)\not\approx H_1(B_{n+1}(S^2);\Z)$. The argument for open manifolds does not apply in this case because the stabilization map $B_n(M)\to B_{n+1}(M)$ does not exist (indeed there are topological obstructions to the existence of such a map). Without the ability to stabilize by adding points, we cannot relate the $n$--point configuration space $B_n(M)$ with the $(n+1)$--point configuration space $B_{n+1}(M)$. 

\para{Our approach} In this paper, we analyze instead the \emph{ordered} configuration spaces $C_n(M)$. The key benefit of this approach is by working with ordered configuration spaces, we have a natural map $C_{n+1}(M)\to C_n(M)$ given by forgetting the last point. This induces maps $H^i(C_n(M))\to H^i(C_{n+1}(M))$ for any manifold $M$, whether open or closed. Note that forgetting the last point requires knowing which is the last point, so this map cannot be defined on $B_n(M)$.

The action of $S_n$ on $C_n(M)$ induces an action on the cohomology groups $H^i(C_n(M))$, and the transfer map for the finite cover $C_n(M)\to B_n(M)$ gives a rational isomorphism
\begin{equation}\label{eq:transfer}H^i(B_n(M);\Q)\approx H^i(C_n(M);\Q)^{S_n}\end{equation} between the cohomology of the quotient and the $S_n$--invariant part of the cohomology of the cover. Thus to understand the cohomology of $B_n(M)$, it suffices to understand the cohomology of $C_n(M)$ together with the action of $S_n$ on it. In fact, we will use the representation theory of $S_n$ to relate successive configuration spaces in a way that would be impossible without using the $S_n$ action.

\para{Representation stability} The maps $H^i(C_n(M))\to H^i(C_{n+1}(M))$ are never isomorphisms with any coefficients, even for $n\gg i$, so we cannot hope to prove homological stability for the ordered configuration spaces $C_n(M)$. However, in Church--Farb \cite{CF} we defined \emph{representation stability}, which lets us nevertheless give a meaningful answer to the question: ``What is the stable homology of the ordered configuration space $C_n(M)$?'' 

Representation stability is a notion which applies to a sequence of rational vector spaces $V_n$, each equipped with an action of the corresponding symmetric group $S_n$, with maps $\phi_n\colon V_n\to V_{n+1}$. We say that such a sequence is \emph{consistent} if $\phi_n$ respects the group actions: \[\qquad\qquad\qquad\qquad\qquad\phi_n(\sigma\cdot v)=\sigma\cdot \phi_n(v)\qquad\qquad\text{for any }\sigma\in S_n,\] where we identify $S_n$ as a subgroup of $S_{n+1}$ in the standard way.

There is a classical correspondence between irreducible representations of $S_n$ and partitions $\lambda$ of $n$ (see e.g. \cite[Theorem 4.3]{FH}). A \emph{partition} of $n$ is a sequence $\lambda=(\lambda_1\geq \cdots \geq \lambda_\ell>0)$ with $\lambda_1+\cdots+\lambda_\ell=n$. We write $|\lambda|=n$ or $\lambda\vdash n$, and say that $\ell=\ell(\lambda)$ is the \emph{length} of $\lambda$. We often write $\lambda=(\lambda_1,\ldots,\lambda_\ell)$. The corresponding irreducible representation is denoted $V_\lambda$.

There are many commonalities between the irreducible representations of $S_n$ for different $n$. For example, the standard representation of $S_n$ on $\Q^n/\Q$ is the irreducible representation $V_{(n-1,1)}$ for any $n>1$. Similarly, the representation $\bwedge^3 (\Q^n/\Q)$ is the irreducible representation $V_{(n-3,1,1,1)}$ for any $n>3$. To capture these commonalities, we establish some additional notation.
If $\lambda=(\lambda_1,\ldots,\lambda_\ell)\vdash k$ is any partition of a fixed number $k$, then for any $n\geq k+\lambda_1$ we define the partition $\lambda[n]\vdash n$ by \[\lambda[n]\coloneq (n-k,\lambda_1,\ldots,\lambda_\ell),\] and define $V(\lambda)_n$ to be the irreducible $S_n$--representation \[V(\lambda)_n\coloneq V_{\lambda[n]}.\] Note that the condition $n\geq k+\lambda_1$ guarantees that $\lambda[n]$ is actually a partition. Every irreducible representation of
$S_n$ is of the form $V(\lambda)_n$ for a unique partition
$\lambda$.
Given a representation $V_n$ of $S_n$, we write $c_\lambda(V_n)$ for the multiplicity of $V(\lambda)_n$ in $V_n$.
\begin{definition}[Representation stability]
\label{def:repstability}
  Let $\{V_n\}$ be a consistent sequence of finite-dimensional rational $S_n$--representations.  We
  say that the sequence $\{V_n\}$ is \emph{uniformly representation stable with stable range $n\geq N$} if for all $n\geq N$, each of the following conditions holds:
\begin{enumerate}[I.]
\item \textbf{Injectivity:} The maps $\phi_n\colon V_n\to V_{n+1}$ are injective.
\item \textbf{Surjectivity:} As an $S_{n+1}$--module, $V_{n+1}$ is spanned by $\phi_n(V_n)$.
\item \textbf{Uniform multiplicity stability:} For each partition $\lambda$, the multiplicity $0\leq c_{\lambda}(V_n)<\infty$ of the irreducible representation $V(\lambda)_n$ in $V_n$ is independent of
  $n$ for all $n\geq N$.
\end{enumerate}
\end{definition}
We say that $\{V_n\}$ is \emph{uniformly multiplicity stable with stable range $n\geq N$} if Condition III holds for all $n\geq N$.
The foundations of representation stability can be found in Church--Farb \cite{CF}, including basic properties, theorems, and applications.\pagebreak

\para{Monotonicity} The central innovation of this paper is the notion of \emph{monotonicity} for a sequence of $S_n$--representations.
\begin{definition}[Monotonicity]
\label{def:monotone}
A consistent sequence $\{V_n\}$ of $S_n$--representations with maps $\phi_n\colon V_n\to V_{n+1}$ is called \emph{monotone for $n\geq N$} if for any $n\geq N$ and for each subspace $W<V_n$ isomorphic to $V(\lambda)_n^{\oplus k}$, the $S_{n+1}$--span of $\phi_n(W)$ contains $V(\lambda)_{n+1}^{\oplus k}$ as a subrepresentation.
\end{definition}

\para{Cohomology of ordered configuration spaces} The main theorem of this paper is the following. We require our manifolds to be connected and orientable (it is easy to see that stability fails already at $H_0$ for disconnected manifolds). Furthermore, here and throughout the paper, we exclude 1--dimensional manifolds, whose configuration spaces are trivial to compute by hand. Finally, we assume that the manifold has finite-dimensional cohomology; in this paper, we say that $M$ is of \emph{finite type} if $H^*(M;\Q)$ is finite-dimensional. This includes all compact manifolds and the interiors of compact manifolds.
\begin{main}
For any connected orientable manifold $M$ of finite type and any $i\geq 0$, the cohomology groups $\{H^i(C_n(M);\Q)\}$ of the ordered configuration space $C_n(M)$ are monotone and uniformly representation stable, with stable range $n\geq 2i$ if $\dim M\geq 3$ and stable range $n\geq 4i$ if $\dim M=2$.
\end{main}

John Wiltshire--Gordon has recently used Theorem 4.1 of \cite{CF} to prove that a stable range of $n\geq 2i$ is not possible for $C_n(\R^2)$. Theorem~\theoremmain\ implies that the $i$--th cohomology $H^i(C_n(M);\Q)$ has a uniform description in terms of representations of $S_n$, with this description independent of $n$ once $n$ is large enough. This description can be thought of as the ``stable cohomology'' of $C_n(M)$, though technically it is not a vector space or a single object at all.

This theorem has concrete implications for which irreducible representations can occur in $H^i(C_n(M);\Q)$. Assuming $\dim M\geq 3$, Theorem~\theoremmain\ implies that $V_\mu$ can only occur in $H^i(C_n(M);\Q)$ if the partition $\mu\vdash n$ satisfies $2\mu_2+\mu_3+\cdots+\mu_\ell\leq 2i$. For large $n$ this is a very strong restriction, implying that the vast majority of irreducible $S_n$--representations do not occur in $H^i(C_n(M);\Q)$.

As a corollary of monotonicity, we immediately obtain a stability result for inclusions of configuration spaces. Any inclusion  $M\hookrightarrow N$ induces an inclusion $C_n(M)\hookrightarrow C_n(N)$, yielding a restriction map $H^i(C_n(N);\Q)\to H^i(C_n(M);\Q)$.
\begin{thmtwo}
For any inclusion $M\hookrightarrow N$ of manifolds as in Theorem~\theoremmain, both the kernel and the image of the induced map on cohomology $H^i(C_n(N);\Q)\to H^i(C_n(M);\Q)$ are uniformly representation stable.
\end{thmtwo}
The corollary is most interesting when $M$ and $N$ have the same dimension. Indeed when their dimensions differ, the map $H^*(C_n(N);\Q)\to H^*(C_n(M);\Q)$ can be completely understood (see Remark~\ref{remark:diffdim}). The proof of the main theorem makes up Section~\ref{section:main}.\\

\para{Applications to unordered configuration spaces} For all $n$, the representation $V(0)_n$ denotes the trivial representation of $S_n$. Thus as a special case of Theorem~\theoremmain, we conclude that the multiplicity of the trivial representation in $H^i(C_n(M);\Q)$ is eventually constant. Applying the transfer isomorphism \eqref{eq:transfer}, we obtain rational homological stability for the unordered configuration spaces $B_n(M)$ in Corollary~\corollarymain. For this special case, we improve the stable range to $n>i$, with no restrictions on dimension, in Section~\ref{section:unordered}.

For closed manifolds, there is no continuous map from $B_n(M)$ to $B_{n+1}(M)$, or vice versa, which induces the desired isomorphism $H_i(B_n(M);\Q)\approx H_i(B_{n+1}(M);\Q)$ of Corollary~\corollarymain. However, there is a \emph{correspondence} (i.e.{} a multi-valued function) $\pi_n\colon B_{n+1}(M)\rightrightarrows B_n(M)$ sending a set of $n+1$ points to all possible $n$--element subsets, and this induces a well-defined map $(\pi_n)_*\colon H_i(B_{n+1}(M);\Q)\to H_i(B_n(M);\Q)$ on homology. As we will see in Section~\ref{section:unordered}, monotonicity for $\{H^i(C_n(M);\Q)\}$ is exactly what we need to conclude that $(\pi_n)_*$ is an isomorphism.
\begin{maincor}
For any connected orientable manifold $M$ of finite type, the unordered configuration spaces $B_n(M)$ satisfy rational homological stability: for each $i\geq 0$ and each $n>i$, the induced map $(\pi_n)_*\colon H_i(B_{n+1}(M);\Q)\overset{\approx}{\longrightarrow} H_i(B_n(M);\Q)$ is an isomorphism.
\end{maincor}

\begin{cortwo}
For any inclusion $M\hookrightarrow N$ of such manifolds, the induced map on the homology of the unordered configuration spaces is stable: for each $i\geq 0$, the rank of the induced map $H_i(B_n(M);\Q)\to H_i(B_{n}(N);\Q)$ is constant for all $n>i$.
\end{cortwo}

One of the most interesting aspects of Corollary~\corollarymain\ is that 
rational homological stability holds even in cases when the integral homology is \emph{not} stable (we saw this above for the configuration space $B_n(S^2)$ of the 2--sphere). The only other example I am aware of where this phenomenon arises naturally is in the twisted cohomology of closed mapping class groups. Looijenga \cite[Theorem 1.1]{Lo} proved that the twisted rational homology groups $H_i(\Mod(S_g);H_1(S_g;\Q))$ stabilize. However, Morita \cite[Corollary 5.4]{Mo} computed that for closed surfaces we have $H_1(\Mod(S_g);H_1(S_g;\Z))\approx \Z/(2-2g)\Z$, so integral stability fails.

\para{Bounds on the stable range} Even for open manifolds, where integral stability holds, the stable range in Corollary~\corollarymain\ of $n>i$ for rational homology is a significant improvement on the best-known bounds for integral stability: for general open manifolds a stable range of $n\geq 2i+2$ was established by Segal \cite[Proposition A.1]{Se}, and this was improved to $n\geq 2i$ for punctured surfaces by Napolitano \cite[Theorem 1]{Na}. The optimal range for integral stability on open manifolds is not known, but a stable range of $n>i$ is not possible in general: for example, from Napolitano's calculations for the punctured torus we have that $H_4(B_5(T^2-\ast);\Z)\not\approx H_4(B_6(T^2-\ast);\Z)$ \cite[Table 3]{Na}.

The stable range of $n>i$ in Corollary~\corollarymain\ is sharp for rational stability. For example, for the 2--torus $T^2$, combining the calculations of Napolitano \cite[Table 2]{Na} with Corollary~\corollarymain\ gives: \begin{itemize}
\item $H_2(B_2(T^2);\Q)=\Q$\ \ while $H_2(B_n(T^2);\Q)=\Q^3$ for $n\geq 3$, \item $H_3(B_3(T^2);\Q)=\Q^4$ while $H_3(B_n(T^2);\Q)=\Q^5$ for $n\geq 4$, \item 
$H_4(B_4(T^2);\Q)=\Q^4$ while $H_4(B_n(T^2);\Q)=\Q^7$ for $n\geq 5$.
\end{itemize}

Corollaries~\corollarymain\ and \corollarytwo, together with the improved stable range, are proved in Section~\ref{section:unordered}. 
We would like to point out that since this paper was first distributed, Randal-Williams \cite[personal communication]{RW} has given another proof of Corollary~\corollarymain, with a stable range of $n\geq 2i+2$ in general.
Finally, the stable range of $n\geq i+1$ in Corollary~\corollarymain\ can be improved if the low-dimensional Betti numbers of $M$ vanish.
\begin{propbetti}
If $b_1(M)=\cdots=b_{k-1}(M)=0$ for some $k<\dim M$, then the stable range for $\{H^i(B_n(M);\Q)\}$ can be improved to $n\geq \frac{i}{k}+1$.
\end{propbetti}



\para{Other configuration spaces} The main theorem also implies stability for certain related configuration spaces. Given a partition $\mu=(\mu_1,\ldots,\mu_k)$, for each $n\geq |\mu|$ we define $B_{n,\mu}(M)$ to be the space of configurations of $n$ points in $M$, of which $\mu_1$ have one color, $\mu_2$ points have another color, etc., and the remaining $n-|\mu|$ points are uncolored. Points which have the same color, or which have no color, are indistinguishable, but points of different colors can be distinguished from each other. Thus $B_{n,\mu}(M)$ is the quotient of the ordered configuration space $C_n(M)$ by the Young subgroup $S_{\mu_1}\times \cdots\times S_{\mu_k}\times S_{n-|\mu|}$. The following theorem is proved in Section~\ref{section:unordered}. A motivic analogue of this theorem when $M$ is an algebraic variety has been proved by Vakil--Wood \cite{VW}.

\begin{thmvw}
For any connected orientable manifold $M$ of finite type, the configuration spaces $B_{n,\mu}(M)$ satisfy rational homological stability: for each $i\geq 0$ and $n\gg i$, there is an isomorphism $H_i(B_{n,\mu}(M);\Q)\approx H_i(B_{n+1,\mu}(M);\Q)$, with a stable range of $n\geq \max(2i,2|\mu|)$ if $\dim M\geq 3$ and a stable range of $n\geq \max(4i,2|\mu|)$ if $\dim M=2$.
\end{thmvw}

Another variant of $C_n(M)$ which has been much studied is $C_n(M,X)$, the configuration space of $n$ ordered points in $M$, each labeled by a point of an auxiliary space $X$. Since the labels are independent, we have $C_n(M,X)\simeq C_n(M)\times X^n$, with $S_n$ acting diagonally. The quotient $B_n(M,X)\coloneq C_n(M,X)/S_n$ is the configuration space of $n$ unordered points in $M$ labeled by $X$. By modifying the proof of Theorem~\theoremmain, it is possible to prove representation stability for $\{H^i(C_n(M,X);\Q)\}$, and thus to deduce rational homological stability for $B_n(M,X)$, as long as $M$ and $X$ are of finite type and $X$ is path-connected (though we do not include the details here). Homological stability for $B_n(M,X)$ was previously known (in integral homology) in the case when $M$ is open, but was not known when $M$ is closed. We hope to revisit these results in a future paper.

\para{Outline of proof}
We prove Theorem~\theoremmain\ by analyzing the Leray spectral sequence of the inclusion $C_n(M)\hookrightarrow M^n$. In the beautiful paper \cite{To}, Totaro completely described the $E_2$ page of this spectral sequence whenever $M$ is an orientable manifold. We recall his description in Section~\ref{section:configleray}. This spectral sequence converges to $H^*(C_n(M);\Q)$.

The main obstacle is to relate the spectral sequence associated to $C_n(M)\hookrightarrow M^n$ with the spectral sequence associated to $C_{n+1}(M)\hookrightarrow M^{n+1}$. In Church--Farb \cite{CF}, representation stability was proved for the cohomology of pure braid groups (that is, for the configuration space $C_n(\R^2)$). Using the ideas introduced there, we prove that the $E_2$ page of the Leray spectral sequence is representation stable. However, the techniques of \cite{CF} are not enough to conclude representation stability for $H^*(C_n(M))$, or even for later pages in the spectral sequence.

Using Theorem~\ref{theorem:pieri} (see below), we can prove that the $E_2$ page of the Leray spectral sequence is \emph{monotone} in the sense of Definition~\ref{def:monotone}.
Monotonicity is the key property that lets us promote representation stability from the $E_2$ page to the $E_3$ page, on to the $E_\infty$ page, and from there to the cohomology of $C_n(M)$. 
We emphasize that very little is known in general about how long this spectral sequence takes to degenerate, or what it converges to (see Remark~\ref{remark:FT}, for example).


\para{Monotonicity of induced representations}
The key technical step in the proof of Theorem~\theoremmain\ turns out to be the following proposition, which belongs purely to the representation theory of the symmetric groups.
Fix an integer $k\geq 0$ and a representation $V$ of $S_k$. We write $V\boxtimes \Q$ for the same vector space considered as a representation of $S_k\times S_\ell$ by defining $S_\ell$ to act trivially.
There is a natural inclusion \[\Ind_{S_k\times S_\ell}^{S_{k+\ell}} V\boxtimes \Q\hookrightarrow\Ind_{S_k\times S_{\ell+1}}^{S_{k+\ell+1}} V\boxtimes \Q\] making this into a consistent sequence of $S_{k+\ell}$--representations. The branching rule  (Proposition~\ref{prop:branching}) implies that this sequence is uniformly representation stable. What we will need is the following theorem, which is significantly stronger.
\begin{theorem:tech}
The sequence $\{\Ind_{S_k\times S_\ell}^{S_{k+\ell}} V\boxtimes \Q\}_\ell$ of $S_{k+\ell}$--representations is monotone and uniformly representation stable.
\end{theorem:tech}

Note that the most natural explanation for monotonicity would be that the $S_{k+\ell+1}$--span of each subrepresentation $V(\mu)_{k+\ell}$ is exactly $V(\mu)_{k+\ell+1}$. However, this is never the case.
The proof of Theorem~\ref{theorem:pieri} requires delving deeply into the explicit construction of the irreducible representations of $S_n$ and the concrete demonstration of the branching rule. An overview of the necessary representation theory, and the proof of Theorem~\ref{theorem:pieri}, make up Section~\ref{section:final}.

\para{Shortcomings} There is one main drawback to the approach taken in this paper: since we use the transfer map and the representation theory of the symmetric group, we can only draw conclusions about rational homology and cannot extend our results to integral homology. However, since Corollary~\corollarymain\ is false for integral homology, this drawback is in some sense unavoidable.

We also required in the main theorems that our manifolds be connected. Again this is unavoidable: for example, if $M$ has $k$ components, $\dim H^0(B_n(M))$ is the number of partitions of $n$ into $k$ nonnegative numbers, which grows without bound if $k>1$. The same explosion of $H^i(B_n(M))$ occurs in general for all $i$. Segal \cite[Proposition A.1]{Se} writes ``Of course we are assuming $M$ is connected.'' In McDuff \cite{McD} this assumption is not made explicit in the statements of the main theorems, but it is used in the proofs.

Finally, we restricted our attention to manifolds whose cohomology is finite-dimensional. This decision is partly aesthetic: saying that two multiplicities coincide conveys much less information when both are infinite. However, this hypothesis is used throughout the arguments in Section 2, and thus in the proofs of the main theorems. Since every manifold can be exhausted by submanifolds which are the interiors of compact manifolds, using Corollary~\theoremtwo\ it should be possible to recover many of the results of this paper for general manifolds (at least for homology). However, I do not know whether all results hold exactly as stated in such generality.

\para{Acknowledgements} Above all, I thank Benson Farb, who was closely involved during the development of these ideas but declined to be listed as a coauthor on this paper. Without his unflagging interest and support, this work would never have been begun, pursued, or completed. I am very grateful to Jordan Ellenberg for pointing out a mistake in an earlier version of this paper. I thank David Kazhdan for a very helpful conversation, and Ulrike Tillmann, Ravi Vakil, and Melanie Matchett Wood for suggesting additional applications.
Finally, I am indebted to the anonymous referee for their careful reading and excellent suggestions.

\section{Representation stability and monotonicity}
\label{section:repstability}
In the remainder of the paper, all vector spaces and all representations are over $\Q$, and all homology and cohomology groups are with $\Q$ coefficients.

\subsection{Representation stability}
\para{Representation stability} The notion of representation stability was defined, for many different families of groups, in Church--Farb \cite[Definition 2.6]{CF}, along with a number of variants. In this paper we will only need the definition of uniform representation stability for $S_n$--representations, which was given in Definition~\ref{def:repstability}. We point out one important change from the definition in \cite{CF}: in this paper we take all representations to be finite-dimensional, so that the multiplicities $c_\lambda(V_n)$ are always finite.
The following proposition is immediate, but we will use it repeatedly, so we state it explicitly.
A sequence of subspaces $\{W_n\}<\{V_n\}$ is a sequence $\{W_n\}$ so that for each $n$ we have $W_n<V_n$ and $\phi_n(W_n)\subset W_{n+1}$.
\begin{proposition}
\label{prop:unifsum} Given $\{W_n\}<\{V_n\}$, if both $\{V_n\}$ and $\{W_n\}$ are uniformly multiplicity stable for $n\geq N$, then $\{V_n/W_n\}$ is too. Conversely, if $\{W_n\}$ and $\{V_n/W_n\}$ are uniformly multiplicity stable for $n\geq N$, then $\{V_n\}$ is too.
\end{proposition}
\begin{proof}
We have $c_\lambda(V_n)=c_\lambda(W_n)+c_\lambda(V_n/W_n)$ for all $\lambda$. If any two terms are constant for $n\geq N$, the third is constant as well.
\end{proof}

\para{Properties of monotone sequences}
In Definition~\ref{def:monotone} we introduced the notion of a \emph{monotone sequence}. This is a substitute for the ``type-preserving'' condition introduced in Church--Farb \cite[Definition 2.12]{CF}, which proved extremely useful in analyzing sequences of representations of algebraic groups. 
We begin by establishing some basic properties of monotone sequences.

In this paper, all the sequences we consider will be uniformly representation stable. We emphasize, however, that monotonicity can be usefully applied to sequences which are not stable, or which have not yet stabilized. As we will see in Proposition~\ref{prop:monsingle}, monotonicity has useful implications even for a single irreducible representation.\\

Given a consistent sequence $\{V_n\}$ with maps $\phi_n\colon V_n\to V_{n+1}$, if $W$ is a subspace of $V_n$, we write $S_{n+1}\cdot W$ for the $S_{n+1}$--span of $\phi_n(W)$. 
We will not be strict in always noting the range $n\geq N$; in each result in this section, the range is the same for all sequences unless otherwise specified.
\begin{proposition}
\label{prop:monsurj}
A sequence $\{V_n\}$ which is monotone and uniformly multiplicity stable for $n\geq N$ is uniformly representation stable for $n\geq N$.
\end{proposition}
\begin{proof}
We are given uniform multiplicity stability (Condition III). The definition of monotonicity implies injectivity (Condition I), since monotonicity fails for any $W\approx V(\lambda)_n$ contained in the kernel of $\phi_n\colon V_n\to V_{n+1}$. It remains to check surjectivity (Condition II). Let $k$ be the stable multiplicity of $V(\lambda)_n$ in $V_n$ for $n\geq N$. By uniform multiplicity stability, the $V(\lambda)_{n+1}$--isotypic component is $V(\lambda)_{n+1}^{\oplus k}$. But monotonicity implies that the image $S_{n+1}\cdot V(\lambda)_n^{\oplus k}$ contains at least $V(\lambda)_{n+1}^{\oplus k}$. Thus each isotypic component of $V_{n+1}$ is contained in $S_{n+1}\cdot V_n$, so $S_{n+1}\cdot V_n=V_{n+1}$ as desired.
\end{proof}

\begin{proposition}
\label{prop:monquot}
If $\{V_n\}$ is monotone and $\{W_n\}<\{V_n\}$, then $\{W_n\}$ is monotone. If in addition $\{W_n\}$ is uniformly multiplicity stable with stable range $n\geq N$, then $\{V_n/W_n\}$ is monotone for $n\geq N$.
\end{proposition}
\begin{proof}
The first claim is trivial: every subrepresentation of $W_n$ is contained in $V_n$, so monotonicity for $\{V_n\}$ applies.
For the second claim, consider a subrepresentation $U<V_n/W_n$ with $U\approx V(\lambda)_n^{\oplus k}$. Let $\widetilde{U}<V_n$ be the preimage of $U$. Recall that $c_\lambda(W_n)$ is the multiplicity of $V(\lambda)_n$ in $W_n$, and note that \[c_\lambda(\widetilde{U})=k+c_\lambda(W_n).\] By monotonicity of $\{V_n\}$, the multiplicity $c_\lambda(S_{n+1}\cdot \widetilde{U})$ of $V(\lambda)_{n+1}$ in $S_{n+1}\cdot \widetilde{U}$ is at least $k+c_\lambda(W_n)$. But uniform stability of $\{W_n\}$ implies that $c_\lambda(W_{n+1})=c_\lambda(W_n)$. It follows that the projection of $S_{n+1}\cdot \widetilde{U}$ to $V_{n+1}/W_{n+1}$---that is, $S_{n+1}\cdot U$---contains $V(\lambda)_{n+1}$ with multiplicity at least $k$, as desired.
\end{proof}

\begin{proposition}
\label{prop:monfilt}
If $\{W_n\}<\{V_n\}$ is a sequence of subspaces so that both $\{W_n\}$ and $\{V_n/W_n\}$ are monotone, then $\{V_n\}$ is monotone. In particular, the direct sum of monotone sequences is monotone.
\end{proposition}
\begin{proof}
Let $U$ be a subrepresentation of $V_n$ with $U\approx V(\lambda)_n^{\oplus k}$. Let $\overline{U}$ be the projection of $U$ to $V_n/W_n$. We have $U\cap W_n\approx V(\lambda)_n^{\oplus \ell}$ and $\overline{U}\approx V(\lambda)_n^{\oplus (k-\ell)}$ for some $\ell$. Thus $S_{n+1}\cdot (U\cap W_n)=(S_{n+1}\cdot U)\cap W_{n+1}$ contains at least $V(\lambda)_{n+1}^{\oplus \ell}$, and $S_{n+1}\cdot \overline{U}=\overline{S_{n+1}\cdot U}$ contains at least $V(\lambda)_{n+1}^{\oplus (k-\ell)}$. Combining these, we conclude that $S_{n+1}\cdot U$ contains at least $V(\lambda)_{n+1}^{\oplus k}$, as desired.
\end{proof}

The following proposition is the most important property of monotone sequences.
\begin{proposition}
\label{prop:monkerim}
Let $\{V_n\}$ and $\{W_n\}$ be monotone sequences for $n\geq N$, and assume that $\{V_n\}$ is uniformly multiplicity stable for $n\geq N$. Then for any consistent sequence of $S_n$--equivariant maps $f_n\colon V_n\to W_n$ (meaning that the following diagram commutes),
\begin{equation}
\label{eq:fn}
\xymatrix{
  V_n\ar^{f_n}[r]\ar_{\phi_n}[d]&W_{n}\ar^{\phi_n}[d]\\
  V_{n+1}\ar_{f_{n+1}}[r]&W_{n+1} }
\end{equation}
the sequences $\{\ker f_n\}$ and $\{\im f_n\}$ are monotone and uniformly multiplicity stable for $n\geq N$.
\end{proposition}
\begin{proof}
First, note that $\{\ker f_n\}$ is a subsequence $\{\ker f_n\}<\{V_n\}$, and similarly $\{\im f_n\}$ is a subsequence $\{\im f_n\}<\{W_n\}$. Indeed if $f_n(x)=0$, the relation \eqref{eq:fn} implies $f_{n+1}(\phi_n(x))=0$, so $\phi(\ker f_n)\subset \ker f_{n+1}$. Similarly, if $f_n(x)=y$, then $f_{n+1}(\phi_n(x))=\phi_n(y)$, so $\phi(\im f_n)\subset \im f_{n+1}$. Thus monotonicity for $\{\ker f_n\}$ and $\{\im f_n\}$ follows from Proposition~\ref{prop:monquot}. Monotonicity implies that the multiplicity $c_\lambda(\ker f_n)$ is nondecreasing, and also that the multiplicity $c_\lambda(\im f_n)$ is nondecreasing. But we have \[c_\lambda(\ker f_n)+c_\lambda(\im f_n)=c_\lambda(V_n),\] and $c_\lambda(V_n)$ is constant by uniform stability of $\{V_n\}$. It follows that the multiplicities of $V(\lambda)_n$ in $\ker f_n$ and in $\im f_n$ are constant as well, verifying uniform multiplicity stability. By Proposition~\ref{prop:monsurj}, this implies that $\{\ker f_n\}$ and $\{\im f_n\}$ are uniformly representation stable for $n\geq N$, as desired. 
\end{proof}

Finally, we point out that the proofs in this section involve interactions between $V(\lambda)_n$ and $V(\lambda)_{n+1}$, but never between different irreducibles $V(\lambda)_n$ and $V(\mu)_n$. This gives the following corollary, which we will use in Section~\ref{section:unordered} when we focus on the trivial representation $V(0)_n$.
\begin{proposition}
\label{prop:monsingle}
For a fixed partition $\lambda$, Propositions~\ref{prop:monquot}, \ref{prop:monfilt}, and \ref{prop:monkerim} remain true if ``uniformly multiplicity stable'' is replaced throughout by ``the multiplicity of $V(\lambda)_n$ is stable''.
\end{proposition}
We remark that the proposition holds even if the monotonicity assumption is only known for subrepresentations isomorphic to $V(\lambda)_n^{\oplus k}$.

\subsection{Induced representations}
\label{section:induced} If $G$ is a finite group with subgroup $H$, recall that for any $H$--representation $V$, the induced representation $\Ind_H^G V$ is the $G$--representation defined by the adjuction:
\[\Hom_H(V,W)\approx \Hom(\Ind_H^G V,W)\] That is, for any $G$--representation $W$, an $H$--linear map $V\to W$ uniquely determines and is determined by a $G$--linear map $\Ind_H^G V\to W$. The induced representation can be constructed explicitly as
\[\Ind_H^G V\coloneq V\otimes_{\Q H}\Q G.\]
The following criterion is essentially another definition of the induced representation. We will use this criterion to recognize induced representations that arise naturally in the proofs ahead.
\begin{criterion}
\label{crit:ind}
Let $V$ be a $G$--representation decomposing as a direct sum $V=\bigoplus_i W_i$ whose summands are permuted transitively by the action of $G$, meaning that for any $i$ and any $g\in G$, we have $g(W_i)=W_j$ for some $j$, and for any $i$ we have $g(W_0)=W_i$ for some $g\in G$. Let $H=\Stab_G(W_0)$ and consider $W_0$ as an $H$--representation. Then $V$ is the $G$--representation induced from the $H$--representation $W_0$:
\[V=\Ind_H^G W_0\]
\end{criterion}

\para{The branching rule}
Fix a representation $V$ of $S_k$. For each $n\geq k$ we may consider the $S_n$--representation \[I_n(V)\coloneq \Ind_{S_k\times S_{n-k}}^{S_n}V\boxtimes \Q.\]  Recall that $V\boxtimes \Q$ denotes the extension of $V$ to a representation of $S_k\times S_{n-k}$ by letting $S_{n-k}$ act trivially. These representations fit into a consistent sequence
\[V\to I_{k+1}(V)\to I_{k+2}(V)
\to \cdots\to I_n(V)\overset{\iota_n}{\to} I_{n+1}(V)\to \cdots\]
where the map $\iota_n\colon I_n(V)\to I_{n+1}(V)$ is the composition \begin{equation}\label{eq:iotan}I_n(V)\hookrightarrow \Ind_{S_n}^{S_{n+1}} I_n(V)\twoheadrightarrow I_{n+1}(V).\end{equation} Here the first map is the natural inclusion, while the second is the natural surjection
\[\Ind_{S_n}^{S_{n+1}} I_n(V)
=V\otimes_{\Q(S_k\times S_{n-k}\times S_1)}\Q S_{n+1}\twoheadrightarrow V\otimes_{\Q(S_k\times S_{n-k+1})}\Q S_{n+1}=I_{n+1}(V).
\]
\begin{theorem}
\label{theorem:pieri}
If $V$ is a fixed representation of $S_k$, the sequence of induced representations $\{\Ind_{S_k\times S_{n-k}}^{S_n} V\boxtimes \Q\}$ is monotone for $n\geq k$ and uniformly representation stable for $n\geq 2k$.
\end{theorem}
Given an irreducible representation $V_\lambda$  of $S_k$, the \emph{branching rule} (Proposition~\ref{prop:branching}) describes the decomposition of the $S_n$--representation $\Ind_{S_k\times S_{n-k}}^{S_n} V_\lambda\boxtimes \Q$.
Uniform stability can be deduced from the branching rule, as D. Hemmer explained to us (see \cite{He} or \cite[Theorem 4.2]{CF}). The real content of Theorem~\ref{theorem:pieri} is that this sequence is monotone; this seems to require much deeper analysis, and is proved in Section~\ref{section:final}.

\section{Ordered configuration spaces}
\label{section:main}
In this section we prove the main theorem of the paper.
\begin{main}
For any connected orientable manifold $M$ of finite type and any $i\geq 0$, the cohomology groups $\{H^i(C_n(M);\Q)\}$ of the ordered configuration space $C_n(M)$ are monotone and uniformly representation stable, with stable range $n\geq 2i$ if $\dim M\geq 3$ and stable range $n\geq 4i$ if $\dim M=2$.
\end{main}

For the rest of the paper, all cohomology is taken with rational coefficients unless otherwise specified, and all tensor products are over $\Q$. Let $M$ be a connected orientable manifold, and let $d=\dim M$ be the dimension of $M$. We assume that $M$ is of \emph{finite type}, meaning that $M$ has finite-dimensional cohomology; this holds in particular for compact manifolds and the interiors of compact manifolds.

\subsection{Outline of proof}
The proof of Theorem~\theoremmain\ occupies the rest of Section~\ref{section:main}, and consists of a number of steps. We begin by analyzing the Leray spectral sequence associated to the inclusion $C_n(M)\hookrightarrow M^n$. We denote the $(p,q)$--entry of the $r$th page by $E^{p,q}_r(n)$, and the differential $\partial$ by $\partial^{p,q}_r(n)\colon E^{p,q}_r(n)\to E^{p+r,q-r+1}_r(n)$. The first page is $E_2^{p,q}(n)$, and the first step is to understand the $E_2$ page and prove representation stability for $\{E_2^{p,q}(n)\}$.
\begin{step1}
For any $p\geq 0$, $q\geq 0$, the sequence $\{E_2^{p,q}(n)\}$ is uniformly representation stable with stable range $n\geq 4\frac{q}{d-1}+2p$.
\end{step1}
We will also need the monotonicity of this first page.
\begin{step2}
For any $p\geq 0$, $q\geq 0$, the sequence $\{E_2^{p,q}(n)\}$ is monotone in the sense of Definition~\ref{def:monotone} for $n\geq 1$.
\end{step2}
The next step is to prove by induction on $r$ that each page of the spectral sequence is uniformly representation stable.
\begin{step3}
For any $p\geq 0$, $q\geq 0$, $r\geq 0$, the sequence $\{E_r^{p,q}(n)\}$ is uniformly representation stable for $n\geq 2(p+q)$ and monotone for $n\geq 2(p+q-1)$ if $\dim M\geq 3$, and uniformly representation stable for $n\geq 4(p+q)$ and monotone for $n\geq 4(p+q-1)$ if $\dim M=2$.
\end{step3}

We now prove representation stability as in the main theorem.
\begin{step4}
For any $i\geq 0$, the sequence of cohomology groups $\{H^i(C_n(M))\}$ is monotone and uniformly representation stable, with stable range $n\geq 2i$ if $\dim M\geq 3$, and with stable range $n\geq 4i$ if $\dim M=2$. 
\end{step4}
Finally, we deduce Corollary~\theoremtwo\ from Theorem~\theoremmain.
\begin{step5}
For any inclusion $\iota\colon M\hookrightarrow N$ inducing $\iota^*\colon H^i(C_n(N))\to H^i(C_n(M))$, the sequences $\{\ker \iota^*\}$ and $\{\im \iota^*\}$ are monotone and uniformly representation stable, with stable range $n\geq 2i$ if $\dim N\geq 3$ and $\dim M\geq 3$, and with stable range $n\geq 4i$ otherwise.
\end{step5} 

\begin{remark}
\label{remark:FT}
Totaro \cite[Theorem 3]{To} proved that if $M$ is a smooth complex projective variety, then the Leray spectral sequence for $C_n(M)\hookrightarrow M^n$ has a single nontrivial differential and degenerates thereafter. However, Felix--Thomas \cite{FT2} have constructed examples of closed manifolds for which the Leray spectral sequence does not degenerate after the first nontrivial differential.
\end{remark}

\subsection{Configuration spaces and the Leray spectral sequence}
\label{section:configleray}
\para{Leray spectral sequence} The Leray spectral sequence for the inclusion $\iota\colon C_n(M)\hookrightarrow M^n$ has $E_2$ page given by
\[E_2^{p,q}=H^p(M^n;R^q\iota_*\underline{\Q})\] and converges to $H^{p+q}(C_n(M);\underline{\Q})=H^{p+q}(C_n(M))$. Here $R^q\iota_*\underline{\Q}$ is the higher direct image sheaf, which in this case is just the sheaf associated to the presheaf
\[U\mapsto H^q(U\cap C_n(M)).\] Totaro \cite{To} found an explicit description of the cohomology groups $H^p(M^n;R^q\iota_*\underline{\Q})$, which is fundamental in the proof of Theorem~\theoremmain\ and is described below.

By the naturality of the Leray spectral sequence, the action of $S_n$ on the pair $(M^n,C_n(M))$ induces an action of $S_n$ on the corresponding spectral sequence; that is, an action on each entry $E_r^{p,q}$ commuting with the differential and inducing the action on $E_{r+1}^{p,q}$. Similarly, the forgetful map of pairs $(M^{n+1},C_{n+1}(M))\to (M^n,C_n(M))$ induces a map of spectral sequences. The fact that this map commutes with the action of $S_n$ (thought of as a subgroup of $S_{n+1}$ on the left side) implies that each entry $E_r^{p,q}$ yields a consistent sequence of $S_n$--representations in the sense of Definition~\ref{def:repstability}. The decompositions constructed in the proof of Step 1 will be preserved by the map $E_2^{p,q}(n)\to E_2^{p,q}(n+1)$, and thus give a decomposition into consistent sequences of representations.

\para{Sub-configuration spaces and diagonals}
 For each partition $\cS$ of $\{1,\ldots,n\}$ into subsets, let $C_\cS(M)$ be the subset of $M^n$ where no coordinates within the same subset coincide. Given such a partition $\cS$ of $\{1,\ldots,n\}$ into subsets, let $|\cS|$ be the number of subsets, and $\overline{\cS}$ the corresponding partition of $n$. For example, if $\cS$ is the partition $\{\{1,2\},\{3,4\},\{5,6\}\}$ of $\{1,\ldots,6\}$, we have $|\cS|=3$, $\overline{\cS}=(2,2,2)$, and
\[C_\cS(M)\coloneq \left\{(x_1,\ldots,x_6)\in M^6\big|x_1\neq x_2, x_3\neq x_4,x_5\neq x_6\right\}.\]
Such a space is the product of smaller configuration spaces. For example, in the case above $C_\cS(M)$ can be identified with $C_2(M)\times C_2(M)\times C_2(M)$. Conversely, let $M^{\cS}$ be the subspace where any two coordinates within the same subset of $\cS$ are required to be equal. In the example above, $M^{\cS}$ would be
\[M^{\cS}\coloneq \left\{(x_1,\ldots,x_6)\in M^6\big|x_1=x_2, x_3=x_4,x_5=x_6\right\}.\]
For future reference, given $\cS$, let $\cS_1$ be the union of the singleton sets in $\cS$, and let $\cS_{\geq 2}$ be $\cS$ with all singleton sets removed. For example, if $\cS=\{\{1,2,4\},\{5,7\},\{3\},\{6\}\}$, then $\cS_1=\{3,6\}$ and $\cS_{\geq 2}=\{\{1,2,4\},\{5,7\}\}$. Finally, if $\cS$ is a partition of $\{1,\ldots,n\}$, let $\cS[n+1]$ be the partition of $\{1,\ldots,n+1\}$ defined by $\cS[n+1]\coloneq \cS\cup\{\{n+1\}\}$.

\para{Euclidean configuration spaces} Arnol'd \cite{Ar} computed the cohomology $H^*(C_n(\R^2))$, and Brieskorn \cite[Lemma 3]{Br} proved that it naturally splits as a direct sum over sub-configuration spaces, as follows. It is not hard to check that the cohomological dimension of $C_n(\R^2)$ is $n-1$, so the cohomological dimension of $C_\cS(\R^2)$ is $n-|\cS|$. Since $C_n(\R^2)$ is contained in $C_\cS(\R^2)$, there is a natural map $H^*(C_\cS(\R^2))\to H^*(C_n(\R^2))$ which turns out to be an embedding. (In some sense, this fact is the foundation of our entire approach to the cohomology of configuration spaces.) The cohomology of the configuration space $C_n(\R^2)$ splits as
\[H^*(C_n(\R^2))=\bigoplus_\cS H^{\ttop}(C_\cS(\R^2))=\bigoplus_\cS H^{n-|\cS|}(C_\cS(\R^2)).\]
For higher-dimensional Euclidean configuration spaces, we have the same decomposition, the only difference being that $C_n(\R^d)$ has cohomological dimension  $(d-1)(n-1)$:
\[H^*(C_n(\R^d))=\bigoplus_\cS H^{\ttop}(C_\cS(\R^d))=\bigoplus_\cS H^{(d-1)(n-|\cS|)}(C_\cS(\R^d)).\]
The action of $S_n$ on $H^*(C_n(\R^d))$ respects this decomposition, permuting the summands according to the action of $S_n$ on partitions $\cS$ of $\{1,\ldots,n\}$. Similarly, the natural map $H^*(C_n(\R^d))\to H^*(C_{n+1}(\R^d))$ respects the decomposition, taking the summand corresponding to $\cS$ to the summand corresponding to $\cS[n+1]$. We will use the following observation in the proof of Steps 1 and 2: singleton sets in $\cS$ correspond to factors of the form $C_1(\R^d)\simeq \R^d$ and thus contribute no cohomology. It follows that any element of $S_n$ which only permutes singleton sets in $\cS$ acts trivially on the summand $H^{(d-1)(n-|\cS|)}(C_\cS(\R^d))$.

\para{Totaro's description of $E_2^{*,*}$} For a $d$--dimensional orientable manifold $M$, Totaro \cite[Theorem 1]{To} proved that the $E_2$ page of the Leray spectral sequence has a similar decomposition. First, ignoring the bigrading, he proved that
\[\bigoplus E_2^{*,*}(C_n(M)\hookrightarrow M^n)=\bigoplus_\cS H^{(d-1)(n-|\cS|)}(C_\cS(\R^d))\otimes H^*(M^\cS).\] The bigrading is such that $H^p(M^\cS)$ has bigrading $(p,0)$ and $H^{(d-1)(n-|\cS|)}(C_\cS(\R^d))$ has bigrading $(0,(d-1)(n-|\cS|))$. Thus the individual terms $E_2^{p,q}$ vanish unless $q$ is divisible by $d-1$, in which case we have (note the reindexing) 
\begin{equation}
\label{eq:e2decomp}
E_2^{p,q(d-1)}=\bigoplus_{\substack{\cS\text{ with}\\|\cS|=n-q}} H^{q(d-1)}(C_\cS(\R^d))\otimes H^p(M^\cS).
\end{equation}
The action of $S_n$ respects this decomposition, permuting the summands according to the action of $S_n$ on partitions $\cS$ of $\{1,\ldots,n\}$ into $n-q$ subsets. The map $E_2(n)\to E_2(n+1)$ also respects this decomposition, taking the summand corresponding to $\cS$ to the summand corresponding to $\cS[n+1]$. The map $H^*(C_\cS(\R^d))\to H^*(C_{\cS[n+1]}(\R^d))$ is the restriction of the natural map $H^*(C_n(\R^d))\to H^*(C_{n+1}(\R^d))$. The map $H^p(M^\cS)\to H^p(M^{\cS[n+1]})$ is similarly induced by forgetting the last point.

\subsection{Stability and monotonicity of $E_2^{p,q}$ (Steps 1 and 2)}
Fix $p$ and $q$ for the rest of this section. Reindexing as in \eqref{eq:e2decomp}, our goal is to prove that the sequence $\{E_2^{p,q(d-1)}(C_n(M)\hookrightarrow M^n)\}$ is monotone and uniformly representation stable. Injectivity follows from the decomposition \eqref{eq:e2decomp} plus the injectivity of the maps $H^*(C_n(\R^d))\to H^*(C_{n+1}(\R^d))$ and $H^p(M^\cS)\to H^p(M^{\cS[n+1]})$. Injectivity for the latter follows from the K\"unneth formula, and injectivity for the former can be seen from the presentation of $H^*(C_n(\R^d))$ (see \cite[Theorem 4.1]{CF}). By Proposition~\ref{prop:monsurj}, it suffices to prove monotonicity and uniform multiplicity stability.

We proceed as follows. First, we will decompose $E_2^{p,q(d-1)}$ as a sum of subrepresentations $E(\mu)_n$, and then further into subrepresentations $E(\mu,r,\alpha)_n$. These subrepresentations will be defined later. We want to show that for each $\mu$, $r$, and $\alpha$ the sequence $\{E(\mu,r,\alpha)_n\}$ is multiplicity stable. To do this, we identify $E(\mu,r,\alpha)_n$ with $\Ind_{S_k\times S_{n-k}}^{S_n}V(\mu,r,\alpha)\boxtimes \Q$, where $V(\mu,r,\alpha)$ is a fixed $S_k$--representation not depending on $n$. Then we appeal to Theorem~\ref{theorem:pieri}, whose proof we defer to Section~\ref{section:final}, to conclude that $\{E(\mu,r,\alpha)_n\}$ is monotone and uniformly multiplicity stable. Finally, Proposition~\ref{prop:unifsum} and Proposition~\ref{prop:monfilt} imply that $\{E_2^{p,q(d-1)}(n)\}$ is monotone and uniformly multiplicity stable, as desired.

\para{Decomposing $E_2^{p,q(d-1)}$ into subrepresentations} The decomposition we describe in the next two sections of the proof holds for all $n$, but not all summands occur when $n$ is small. So that all summands are guaranteed to occur, we assume for simplicity that $n\geq p+2q$, but we emphasize that the description below applies for all $n\geq 1$.

Given a partition $\cS$ of $\{1,\ldots,n\}$ into subsets, recall that $\overline{\cS}$ is the partition of $n$ encoding the sizes of those subsets. Fix a partition $\mu=(\mu_1\geq \cdots\geq \mu_m)$ with $\mu\vdash q$. Throughout the proof, $m$ will always be $\ell(\mu)$, the number of entries of $\mu$; in particular, $\mu_m>0$. Given such a $\mu$, we define the partition \[\mu\langle n\rangle=(\mu_1+1\geq \mu_2+1\geq \ldots\geq\mu_m+1\geq 1\geq \ldots\geq 1)\] with $n-q$ entries, so that $\mu\langle n\rangle\vdash n$ and $\ell(\mu\langle n\rangle)=n-q$. (Note that $m\leq q$, so our assumption gives $n-q\geq p+q\geq q\geq m$.) Let $E(\mu)_n$ be the subspace of $E_2^{p,q(d-1)}$ defined by
\[E(\mu)_n=\bigoplus_{\substack{\cS\text{ with}\\\overline{\cS}=\mu\langle n\rangle}} H^{q(d-1)}(C_\cS(\R^d))\otimes H^p(M^\cS).\] The action of $S_n$ on partitions $\cS$ of $\{1,\ldots,n\}$ does not change $\overline{\cS}$, so $E(\mu)_n$ is preserved by the action of $S_n$ on $E_2^{p,q(d-1)}$.
However, we will need to decompose $E(\mu)_n$ further, according to the decomposition of $H^p(M^\cS)$.

\para{Decomposing $E(\mu)_n$ further} Recall that $\cS_1$ is the union of singleton sets in $\cS$, and $\cS_{\geq 2}$ is $\cS$ with singleton sets removed. Note that if $\overline{\cS}=\mu\langle n\rangle$ for $\mu$ as above, we have $|\cS_{\geq 2}|=m$. We have a natural splitting $M^{\cS}=M^{\cS_{\geq 2}}\times M^{\cS_1}$. (Of course we could decompose this product further, but it turns out that for us the important distinction is between $\cS_{\geq 2}$ and $\cS_1$.) For $i\in \cS_1$, denote by $M_i$ the $i$th factor of $M^n$ (or $M^\cS$ or $M^{\cS_1}$, depending on context). In the following formula $r$ ranges from $0$ to $p$, and $a$ ranges over all functions on $\cS_1$ with nonnegative integer values and total sum $p-r$. By ordering the elements of $S_1$ as $i_1,\ldots,i_{n-q-m}$, we can identify such a function with its sequence of values $a=(a_1,\ldots,a_{n-q-m})$; we have $a_j\geq 0$ and $r+\sum a_j=p$. We will see below that the choice of ordering is irrelevant. For future reference, we define $\supp(a)$ to be the subset of $\cS_1$ where $a$ is nonzero.

The K\"unneth formula gives a decomposition
\begin{align*}H^p(M^\cS)&=\bigoplus_{r+r'=p}H^r(M^{\cS_{\geq 2}})\otimes H^{r'}(M^{\cS_1})\\
&=\bigoplus_{r+\sum a_j=p}H^r(M^{\cS_{\geq 2}})\otimes H^{a_1}(M_{i_1})\otimes \cdots\otimes H^{a_{n-q-m}}(M_{i_{n-q-m}})
\end{align*}
We denote by $H^{(r,a)}(M^\cS)$ the summand corresponding to $(r,a)$, so that the above decomposition can be abbreviated simply as
\[H^p(M^\cS)=\bigoplus_{r+\sum a_j=p}H^{(r,a)}(M^\cS).\]

\noindent Given such a function $a$, let $\alpha=\overline{a}$ be the partition  $\alpha=(\alpha_1\geq \cdots\geq \alpha_\ell)$ consisting of the positive values of $a$, arranged in descending order. For example, if $a=(0,0,2,0,3,1,0)$, so that
\begin{equation}
\label{eq:Hma}
H^{(m,a)}(M^\cS)= H^m(M^{\cS_{\geq 2}})\otimes H^2(M_3)\otimes H^3(M_5)\otimes H^1(M_6),
\end{equation} we have $\overline{a}=(3\geq 2\geq 1)$. For the rest of the proof, $\ell$ will be $\ell(\overline{a})$, which is just $|\supp(a)|$, the number of positive values of $a$. Note that since $\sum a_j\leq p$ and $a$ is integer-valued, we always have $\ell\leq p$. Also note that neither $\ell$ nor the partition $\overline{a}$ depends on the ordering on $\cS_1$ we chose earlier. 

For $r$ an integer and $\alpha$ a partition with $0\leq r\leq p$ and $r+|\alpha|= p$, define $H^{(r,\alpha)}(M^\cS)$ to be the subspace of $H^p(M^\cS)$ defined by
\[H^{(r,\alpha)}(M^\cS)=\bigoplus_{\overline{a}=\alpha}H^{(r,a)}(M^\cS).\] When the factors in a product are permuted, the corresponding summands in the K\"unneth formula are similarly permuted (with a sign depending on the grading). It follows that any element of $S_n$ which preserves $\cS$ will preserve the subspace $H^{(r,\alpha)}(M^\cS)$. Furthermore, an element which takes $\cS$ to another partition $\cS'$ of $\{1,\ldots,n\}$ will take $H^{(r,\alpha)}(M^\cS)$ to the subspace $H^{(r,\alpha)}(M^{\cS'})$. This implies that if we define $E(\mu,r,\alpha)_n$ to be the subspace
\[E(\mu,r,\alpha)_n\coloneq \bigoplus_{\overline{\cS}=\mu\langle n\rangle}H^{q(d-1)}(C_\cS(\R^d))\otimes H^{(r,\alpha)}(M^\cS),\] then $E(\mu,r,\alpha)_n$ is preserved by the action of $S_n$. Note that \[E(\mu)_n=\bigoplus_{\substack{0\leq r\leq p\\r+|\alpha|=p}}E(\mu,r,\alpha)_n,\]
and that the set of $r$ and $\alpha$ which occur is finite and independent of $n$. Thus it suffices to show that for fixed $\mu$, $r$, and $\alpha$ the sequence $\{E(\mu,r,\alpha)_n\}$ is monotone and uniformly multiplicity stable for $n\geq 4q+2p$. (This differs from the bound of $n\geq 4\frac{q}{d-1}+2p$ in the statement of the proposition because we are considering $E_2^{p,q(d-1)}$, rather than $E_2^{p,q}$.)

\para{$E(\mu,r,\alpha)_n$ as an induced representation} Fix $\mu$, $r$, and $\alpha$ for the rest of the proof. Recall that $H^{(r,\alpha)}(M^\cS)$ is the direct sum, over functions $a$ on $\cS_1$ with $\overline{a}=\alpha$, of $H^{(r,a)}(M^\cS)$. Thus $E(\mu,r,\alpha)_n$ can be rewritten as
\[E(\mu,r,\alpha)_n\coloneq \bigoplus_{\substack{\overline{\cS}=\mu\langle n\rangle\\\overline{a}=\alpha}}H^{q(d-1)}(C_\cS(\R^d))\otimes H^{(r,a)}(M^\cS).\] From the discussion above, this decomposition is preserved by the action of $S_n$, meaning that each element of $S_n$ either preserves a given summand or takes it to another summand (according to the action of $S_n$ on partitions $\cS$ and on functions $a$ supported on $\cS_1$). Thus applying Criterion~\ref{crit:ind}, if $\T$ is an arbitrary partition with $\overline{\T}=\mu\langle n\rangle$ and $b$ a function on $\T_1$ with $\overline{b}=\alpha$, we have \[E(\mu,r,\alpha)_n=\Ind_{\Stab(\T,b)}^{S_n} H^{q(d-1)}(C_\T(\R^d))\otimes H^{(r,b)}(M^\T)\]
We will later describe a rule for choosing $\T$ and $b$, but for now we allow them to be arbitrary. The subgroup $\Stab(\T,b)<S_n$ stabilizing the partition $\T$ and the function $b$ can be frustrating to describe exactly, but fortunately we will not have to do so. Let us abbreviate $H^{q(d-1)}(C_\T(\R^d))\otimes H^{(r,b)}(M^\T)$ to just $H(\T,r,b)$.

\para{Finding a subgroup acting trivially} By $\bigcup \T_{\geq 2}$ we mean the union of all subsets in $\T_{\geq 2}$, which is exactly $\{1,\ldots,n\}\setminus \T_1$. It suffices to recall that $|\T_{\geq 2}|=m$, so that $|\bigcup\T_{\geq 2}|=q+m$, and that $|\supp(b)|=\ell\leq p$. We can clearly choose $\T$ and $b$ so that
\begin{equation}
\label{eq:Tb}
\bigcup \T_{\geq 2}\cup \supp(b) = \{1,\ldots,q+m+\ell\}.
\end{equation} Define $k\coloneq q+m+\ell$ for the rest of the proof. Consider $S_{n-k}$ as the subgroup of $S_n$ permuting the subset $\{k+1,\ldots,n\}$. Whenever we refer to $S_{n-k}$, we mean this subgroup, and similarly $S_k$ is the subgroup permuting $\{1,\ldots,k\}$. First, note that \[S_{n-k}<\Stab(\T,b)<S_k\times S_{n-k}.\] Indeed \eqref{eq:Tb} implies that $\{k+1,\ldots,n\}\subset \T_1$, so $S_{n-k}<\Stab(\T)$, and futhermore that $\{k+1,\ldots,n\}$ is disjoint from $\supp(b)$, so $S_{n-k}<\Stab(b)$. But both $\T_{\geq 2}$ and $\supp(b)$ are preserved by $\Stab(\T,b)$, so $\Stab(\T,b)$ preserves their union $\{1,\ldots,k\}$, implying $\Stab(\T,b)<\Stab(\{1,\ldots,k\})=S_k\times S_{n-k}$.\pagebreak

The key point is that in fact $S_{n-k}$ acts trivially on $H(\T,r,b)=H^{q(d-1)}(C_\T(\R^d))\otimes H^{(r,b)}(M^\T)$. For the first factor, we have already observed (in the section ``Euclidean configuration spaces'' of Section~\ref{section:configleray}) that any permutation which only permutes elements of $\T_1$ acts trivially on $H^{q(d-1)}(C_\T(\R^d))$, since $\T_1$ consists of exactly those components which do not contribute to cohomology.

For the second factor $H^{(r,b)}(M^\T)$, note that we have not yet used the assumption that $M$ is connected (except implicitly in the example \eqref{eq:Hma}). This assumption implies that $H^0(M)=\Q$. Since $\{k+1,\ldots,n\}$ is disjoint from $\supp(b)$, the contribution of these factors to $H^{(r,b)}(M^\T)$ is just \[\cdots\otimes H^0(M_{k+1})\otimes \cdots\otimes H^0(M_n)=\cdots\otimes\Q\otimes \cdots\otimes \Q=\cdots\otimes\Q.\] In particular, permuting these factors among themselves acts trivially on $H^{(r,b)}(M^\T)$.

\para{$E(\mu,r,\alpha)_n$ as an induced representation, simplified} The fact that $S_{n-k}<\Stab(\T,b)<S_k\times S_{n-k}$ implies that $\Stab(\T,b)=H\times S_{n-k}$ for some $H<S_k$. We define the $S_k$--representation $V(\mu,r,\alpha)$ as follows:
\[V(\mu,r,\alpha)\coloneq \Ind_H^{S_k} H^{q(d-1)}(C_\T(\R^d))\otimes H^{(r,b)}(M^\T)\]
We have written $V(\mu,r,\alpha)$ without a subscript because, as we will see below, $V(\mu,r,\alpha)$ does not depend on $n$.

The reason for focusing on this representation is that it lets us rewrite $E(\mu,r,\alpha)_n$ as
\[E(\mu,r,\alpha)_n=\Ind_{S_k\times S_{n-k}}^{S_n}\ \ V(\mu,r,\alpha)\boxtimes \Q.\] Recall that $V(\mu,r,\alpha)\boxtimes \Q$ is just the $S_k$--representation $V(\mu,r,\alpha)$, extended to a representation of $S_k\times S_{n-k}$ by defining $S_{n-k}$ to act trivially. Then verifying the claim is just a matter of rearranging definitions:
\begin{align*}
E(\mu,r,\alpha)_n&=\Ind_{\Stab(\T,b)}^{S_n} H(\T,r,b)\\
&=\Ind_{H\times S_{n-k}}^{S_n} H(\T,r,b)\boxtimes \Q\\
&=\Ind_{S_k\times S_{n-k}}^{S_n}\left(\Ind_{H\times S_{n-k}}^{S_k\times S_{n-k}} H(\T,r,b)\boxtimes \Q\right)\\
&=\Ind_{S_k\times S_{n-k}}^{S_n}\left(\Ind_{H}^{S_k} H(\T,r,b)\right)\boxtimes \Q\\
&=\Ind_{S_k\times S_{n-k}}^{S_n}V(\mu,r,\alpha)\boxtimes \Q
\end{align*}

\para{Relating $E(\mu,r,\alpha)_n$ to $E(\mu,r,\alpha)_{n+1}$}
Our goal is to verify that the inclusion \[E(\mu,r,\alpha)_n\hookrightarrow E(\mu,r,\alpha)_{n+1}\] coincides with the map \[\iota_n\colon \Ind_{S_k\times S_{n-k}}^{S_n}V(\mu,r,\alpha)\boxtimes \Q\to \Ind_{S_k\times S_{n+1-k}}^{S_{n+1}}V(\mu,r,\alpha)\boxtimes \Q\] described in Section~\ref{section:induced}, so that we can apply Theorem~\ref{theorem:pieri}.\pagebreak

Given $\mu$ and $\alpha$, for each $n$ we chose a partition $\T$ of $\{1,\ldots,n\}$ so that $\overline{\T}=\mu\langle n\rangle$, and a function $b$ on $\T_1$ with $\overline{b}=\alpha$. So far we have not pinned down $\T$ and $b$ exactly, only requiring that $\bigcup\T_{\geq 2}\cup \supp(b)=\{1,\ldots,k\}$. We now describe a global choice of $\T$ and $b$ so that the following properties hold: if $\T$ is the partition of $\{1,\ldots,n\}$ associated to $\mu$, the partition of $\{1,\ldots,n+1\}$ associated to $\mu$ is $\T\langle n+1\rangle$. Furthermore, if $b$ is the function on $\T_1$ associated to $\alpha$, the corresponding function on $\T\langle n+1\rangle_1=\T_1\cup\{n+1\}$ should be the same function extended by 0 on $n+1$. The simplest way to do this explicitly is to ``left-justify'' $\T$ and $b$: for example, if $\mu=(2,2)$, so that $\mu\langle n\rangle=(3,3,1,\ldots,1)$, take \[\T=\{\{1,2,3\},\{4,5,6\},\{7\},\{8\},\ldots,\{n\}\}.\] Then if $\alpha=(4,2)$, define $b$ by $b(7)=4$, $b(8)=2$, and $b(i)=0$ otherwise.

Recall that $\Stab(\T,b)$ splits as a product $H\times S_{n-k}$. The properties just described imply that the action of $H$ on $H(\T,r,b)$ is independent of $n$,  and thus the action of $S_k$ on $V(\mu,r,\alpha)$ is independent of $n$ as well. We now can apply Theorem~\ref{theorem:pieri} to conclude that \{$E(\mu,r,\alpha)_n\}$ is monotone for $n\geq k$ and uniformly representation stable for $n\geq 2k$.

Since $k=q+m+\ell$, $m\leq q$, and $\ell\leq p$, we have $k\geq 2q+p$, so in particular we have stability for $n\geq 4q+2p$, as desired. By Proposition~\ref{prop:unifsum} and Proposition~\ref{prop:monfilt}, this implies that $\{E_2^{p,q(d-1)}(n)\}$ is monotone and uniformly representation stable for $n\geq 4q+2p$. This concludes the proof of Steps 1 and 2.\\

We remark that this bound is sharp for $\mu=(1,\ldots,1)\vdash q$, $r=0$, and $\alpha=(1,\ldots,1)\vdash p$. In this case we have
\[V(\mu,r,\alpha)=H^1(C_2(\R^d))\otimes \cdots\otimes H^1(C_2(\R^d))\otimes H^1(M)\otimes \cdots\otimes H^1(M)\] with the term $H^1(C_2(\R^d))$ occuring $q$ times and the term $H^1(M)$ occuring $p$ times. A summand of this form can only occur for $n\geq 2q+p$, and $\Ind_{S_k\times S_{n-k}}^{S_n} V(\mu,r,\alpha)$ only stabilizes once $n\geq 4q+2p$.

\subsection{Stability and monotonicity of $E_r^{p,q}$ (Step 3)}
\label{section:usingmono}
We first prove that $\{E_r^{p,q}(n)\}$ is monotone and uniformly representation stable for sufficiently large $n$. We then explain how to obtain the desired explicit bound on the stable range.

For the moment we are ignoring the stable range, so every time we write ``monotone'' or ``uniformly representation stable'' in this paragraph, ``for sufficiently large $n$'' is implied. We also use ``uniformly stable'' as an abbreviation for ``uniformly representation stable''. The proof proceeds by induction on $r$. Our base case is $r=2$, in which case Steps 1 and 2 state that for each $p\geq 0$ and $q\geq 0$ the sequence $\{E_2^{p,q}(n)\}$ is monotone and uniformly representation stable. Assume that $\{E_r^{p,q}(n)\}$ is monotone and uniformly representation stable. By Proposition~\ref{prop:monkerim}, the sequences $\{\ker \partial_r^{p,q}(n)\}$ and $\{\im\partial_r^{p,q}(n)\}$ are monotone and uniformly representation stable. The next page $E_{r+1}^{p,q}(n)$ is given by
\[E_{r+1}^{p,q}(n)=\ker\partial_r^{p,q}(n)\ /\ \im\partial_r^{p-r,q+r-1}(n).\] By Proposition~\ref{prop:unifsum}, $\{E_{r+1}^{p,q}(n)\}$ is uniformly multiplicity stable. Furthermore, since both terms are monotone and $\{\im\partial_r^{p-r,q+r-1}(n)\}$ is uniformly stable, Proposition~\ref{prop:monfilt} implies that $\{E_{r+1}^{p,q}(n)\}$ is monotone. By Proposition~\ref{prop:monsurj}, these imply that $\{E_{r+1}^{p,q}(n)\}$ is uniformly representation stable, as desired.\pagebreak

The explicit stable range is provided by the following proposition.
\begin{proposition}
\label{prop:inductivestep}
Given a consistent sequence $\{E_r^{p,q}(n)\}$ of $S_n$-equivariant spectral sequences an integer $\ell$, and a positive real number $m$, if $\{E_2^{p,q}(n)\}$ is uniformly stable for $n\geq m(p+q+\ell)$ and monotone for $n\geq m(p+q+\ell-1)$,
then $\{E_r^{p,q}(n)\}$ is uniformly stable for $n\geq m(p+q+\ell)$ and monotone for $n\geq m(p+q+\ell-1)$ for all $r\geq 2$.
\end{proposition}
The hypothesis of the proposition is provided by Steps 1 and 2, which state that $\{E_2^{p,q}(n)\}$ is uniformly stable for $n\geq 4\frac{q}{d-1}+2p$ and monotone for $n\geq 1$. If $d\geq 3$ we have $4\frac{q}{d-1}+2p\leq 2(p+q)$, so the hypothesis of the proposition holds with $m=2$ and $\ell=0$, while if $d=2$ we have $4\frac{q}{d-1}+2p=4q+2p\leq 4(p+q)$, so the hypothesis of the proposition holds with $m=4$ and $\ell=0$.
\begin{proof}[Proof of Proposition~\ref{prop:inductivestep}]
Assume by induction on $r$ that $\{E_{r}^{p,q}(n)\}$ is uniformly stable for $n\geq m(p+q+\ell)$ and monotone for $n\geq m(p+q+\ell-1)$. The next page is given by
\[E_{r+1}^{p,q}(n)=\ker\partial_{r}^{p,q}(n)\ /\ \im\partial_{r}^{p-r,q+r-1}(n)\]
The inductive hypothesis implies that the domain of $\partial_r^{p-r,q+r-1}(n)$ is uniformly stable for $n\geq m(p+q+\ell-1)$ and monotone for $n\geq m(p+q+\ell-2)$, and its codomain is monotone for $n\geq m(p+q+\ell-1)$. Applying Proposition~\ref{prop:monkerim}, the image $\{\im\partial_{r}^{p-r,q+r-1}(n)\}$ is monotone and uniformly stable for $n\geq m(p+q+\ell-1)$. Similarly, the inductive hypothesis implies that the domain of $\partial_r^{p,q}(n)$ is uniformly stable for $n\geq m(p+q+\ell)$ and monotone for $n\geq m(p+q+\ell-1)$, and its codomain is monotone for $n\geq m(p+q+\ell)$. Applying Proposition~\ref{prop:monkerim}, the kernel $\{\ker\partial_{r}^{p,q}(n)\}$ is monotone and uniformly stable for $n\geq m(p+q+\ell)$. However since $\ker\partial_r^{p,q}(n)$ is contained in $E_r^{p,q}(n)$, we actually have the stronger statement that $\{\ker\partial_r^{p,q}(n)\}$ is monotone for $n\geq m(p+q+\ell-1)$.

Since $\{\ker\partial_r^{p,q}(n)\}$ and $\{\im\partial_r^{p-r,q+r-1}(n)\}$ are both monotone for $n\geq m(p+q+\ell-1)$, and $\{\im\partial_r^{p-r,q+r-1}(n)\}$ is uniformly stable for $n\geq m(p+q+\ell-1)$, Proposition~\ref{prop:monquot} implies that their quotient $\{E_{r+1}^{p,q}(n)\}$ is monotone for $n\geq m(p+q+\ell-1)$. Finally,  $\{\ker\partial_r^{p,q}(n)\}$ and $\{\im\partial_r^{p-r,q+r-1}(n)\}$ are both uniformly stable for $n\geq m(p+q+\ell)$, so Proposition~\ref{prop:unifsum} implies that $\{E_{r+1}^{p,q}(n)\}$ is uniformly stable for $n\geq m(p+q+\ell)$.
\end{proof}

\subsection{Stability and monotonicity of $H^i(C_n(M))$ (Step 4)}
There is a natural filtration of $H^i(C_n(M))$ by subspaces \[0<F^i_i(n)<F^i_{i-1}(n)<\cdots<F^i_1(n)<F^i_0(n)=H^i(C_n(M))\] so that the successive quotients $F^i_p(n)/F^i_{p+1}(n)$ are isomorphic to $E_\infty^{p,i-p}(n)$. The naturality of this filtration implies that it is preserved by the action of $S_n$, and the induced action on the quotient $F^i_p(n)/F^i_{p+1}(n)$ makes it isomorphic to $E_\infty^{p,i-p}(n)$ as an $S_n$--representation. The filtration is also preserved by the map $H^i(C_n(M))\to H^i(C_{n+1}(M))$, and the induced map $E_\infty^{p,i-p}(n)\to E_\infty^{p,i-p}(n)$ agrees with the map induced from $E_2^{p,i-p}(n)\to E_2^{p,i-p}(n)$.\pagebreak

Since the Leray spectral sequence is a first-quadrant spectral sequence, for each $p\geq 0$ and $i\geq 0$ we have $E_\infty^{p,i-p}=E_r^{p,i-p}$ for some finite $r$, so Step 3 implies that $\{E_\infty^{p,i-p}(n)\}$ is monotone and uniformly stable for $n\geq 2(p+(i-p))=2i$ if $\dim M\geq 3$ (resp. for $n\geq 4(p+(i+p))=4i$ if $\dim M=2$). We prove by reverse induction on $p$ that $\{F^i_p(n)\}$ is monotone and uniformly stable for $n\geq 2i$ (resp. $n\geq 4i$). If we assume that $\{F^i_{p+1}(n)\}$ is monotone and uniformly stable for $n\geq 2i$ (resp. $n\geq 4i$), we also have that $\{F^i_p(n)/F^i_{p+1}(n)=E_\infty^{p,i-p}(n)\}$ is monotone and uniformly stable for $n\geq 2i$ (resp. $n\geq 4i$). Thus Proposition~\ref{prop:unifsum} and Proposition~\ref{prop:monfilt} imply that $\{F^i_p(n)\}$ is monotone and uniformly stable for $n\geq 2i$ (resp. $n\geq 4i$), as desired. Since $F^i_0(n)=H^i(C_n(M))$, this completes the proof of Theorem~\theoremmain.

\subsection{Stability of map $H^i(C_n(N))\to H^i(C_n(M))$ (Step 5)}
Any inclusion $\iota\colon M\hookrightarrow N$ induces an inclusion $C_n(M)\hookrightarrow C_n(N)$, which induces a map $\iota^*\colon H^i(C_n(N))\to H^i(C_n(M))$ for each $i$. By Theorem~\theoremmain, both $\{H^i(C_n(N))\}$ and $\{H^i(C_n(M))\}$ are monotone and uniformly representation stable. Applying Proposition~\ref{prop:monkerim} to the map $H^i(C_n(N))\to H^i(C_n(M))$, we obtain the following corollary.

\begin{thmtwo}
For any inclusion $\iota\colon M\hookrightarrow N$ of connected orientable manifolds of finite type inducing $\iota^*\colon H^i(C_n(N))\to H^i(C_n(M))$, the sequences $\{\ker \iota^*\}$ and $\{\im \iota^*\}$ are monotone and uniformly representation stable, with stable range $n\geq 2i$ if $\dim N\geq 3$ and $\dim M\geq 3$, and with stable range $n\geq 4i$ otherwise.
\end{thmtwo}

\begin{remark}
\label{remark:diffdim}
The corollary is most interesting when $M$ and $N$ have the same dimension. If their dimensions are different, then the map $H^*(C_n(N))\to H^*(C_n(M))$ can be completely determined, as we now describe.

For an inclusion $\R^d\hookrightarrow \R^e$ with $d\neq e$, the induced map $H^*(C_n(\R^e))\to H^*(C_n(\R^d))$ is trivial. To see this, note that $H^*(C_n(\R^e))$ is generated by $H^{e-1}(C_n(\R^e))$, which is spanned by the images of all the maps $H^{e-1}(C_2(\R^e))\to H^{e-1}(C_n(\R^e))$ given by the projections onto each two-element subset. For each such projection, the composition \[H^{e-1}(C_2(\R^e))\to H^{e-1}(C_n(\R^e))\to H^{e-1}(C_n(\R^d))\] can be factored instead as \[H^{e-1}(C_2(\R^e))\to H^{e-1}(C_2(\R^d))\to H^{e-1}(C_n(\R^d)).\] But if $d\neq e$, the middle group vanishes, because $C_2(\R^d)$ has cohomology only in degree $d-1$.

This implies in general that if $M$ and $N$ are of different dimensions, the map of $E_2$ pages \[E_2^{p,q}(C_n(N)\hookrightarrow N^n)\to E_2^{p,q}(C_n(M)\hookrightarrow M^n)\] vanishes for $q>0$. This follows from Totaro's description, since any such entry has a $H^*(C_\T(\R^d))$ factor. The map is nontrivial only on the bottom row $E_2^{p,0}$, where we have \[E_2^{p,0}(C_n(N)\hookrightarrow N^n)=H^p(N^n)\to H^p(M^n)=E_2^{p,0}(C_n(M)\hookrightarrow M^n).\] We conclude that the map $H^*(C_n(N))\to H^*(C_n(M))$ vanishes except on classes coming from the ambient space $N^n$, where it is the map induced by the restriction $H^*(N^n)\to H^*(M^n)$.
\end{remark}

\section{Unordered configuration spaces}
\label{section:unordered}
The transfer isomorphism \eqref{eq:transfer} implies that the dimension of  $H_i(B_n(M))$ is the multiplicity of the trivial representation $V(0)_n$ in $H^i(C_n(M))$. Thus it follows immediately from Theorem~\theoremmain\ that the homology $H_i(B_n(M))$ is eventually constant. Similarly it follows from Corollary~\theoremtwo\ that for an inclusion $M\hookrightarrow N$ the rank of $H_i(B_n(M))\to H_i(B_n(N))$ is eventually constant, essentially proving Corollary~\corollarytwo. To complete the proof of Corollary~\corollarymain, we must check that the map $H_i(B_{n+1}(M);\Q)\to H_i(B_n(M);\Q)$ induced by the correspondence $\pi_n\colon B_{n+1}(M)\rightrightarrows B_n(M)$ is an isomorphism.
 
We can also improve the stable range for $H_i(B_n(M))$ to $n\geq i+1$ by a more careful analysis.
In fact, if the low-dimensional Betti numbers of $M$ vanish, we can improve the stable range in Corollary~\corollarymain\ even further.
\begin{proposition}
\label{prop:improvebetti}
If $b_1(M)=\cdots=b_{k-1}(M)=0$ for some $k<\dim M$, then the stable range for $\{H^i(B_n(M);\Q)\}$ can be improved to $n\geq \frac{i}{k}+1$.
\end{proposition}
The assumption that $k<\dim M$ is not restrictive, since any manifold which satisfies this condition for some $k\geq \dim M$ is either rationally acyclic or a rational homology sphere. These cases are completely understood, and in particular in such cases $H^i(B_n(M);\Q)$ stabilizes immediately at $n\geq 2$ for all $i$ (see, e.g., F\'{e}lix--Thomas \cite{FT1}). 
\begin{proof}[Proof of Corollary~\corollarymain\ and Proposition~\ref{prop:improvebetti}.]
We begin by showing that monotonicity for the trivial representation in $\{H^i(C_n(M))\}$ implies that the map $\pi_n^*\colon H^i(B_n(M))\to H^i(B_{n+1}(M))$ is an isomorphism. The covering map $p\colon C_n(M)\to B_n(M)$ induces a correspondence $p^{-1}\colon B_{n+1}(M)\rightrightarrows C_{n+1}(M)$ taking a point to its preimage. This fits into the following diagram, where the top map $C_{n+1}(M)\to C_n(M)$ forgets the last point.
\begin{equation}
\label{eq:corrs}
\begin{array}{c}\xymatrix{
  C_{n+1}(M)\ar[r]&C_{n}(M)\ar^p[d]\\
  B_{n+1}(M)\ar[u]<-.5ex>\ar^{p^{-1}}[u]<.5ex>\ar_{\ \ \pi_n}[r]<-.5ex>\ar[r]<.5ex>&B_{n}(M)}
\end{array}\end{equation} This diagram commutes up to a constant, meaning that the composition of the top three maps coincides with the correspondence $n!\cdot \pi_n$. On cohomology, $p^*$ induces the identification $H^i(B_n(M))\approx (H^i(C_n(M)))^{S_n}$, while the top map induces the standard inclusion $H^i(C_n(M))\hookrightarrow H^i(C_{n+1}(M))$. The map $(p^{-1})^*\colon H^i(C_{n+1}(M))\to H^i(B_{n+1}(M))$ is obtained by averaging over the action of $S_{n+1}$, yielding the projection $H^i(C_{n+1}(M))\twoheadrightarrow (H^i(C_{n+1}(M)))^{S_{n+1}}\approx H^i(B_{n+1}(M))$.
But monotonicity for the trivial representation in $\{H^i(C_n(M))\}$ says exactly that the composition \[(H^i(C_n(M)))^{S_n}\hookrightarrow  H^i(C_{n+1}(M))\twoheadrightarrow (H^i(C_{n+1}(M)))^{S_{n+1}}\] is injective. Thus once the multiplicity of the trivial representation stabilizes, the map $\pi_n^*\colon H^i(B_n(M))\to H^i(B_{n+1}(M))$ is an isomorphism.\\

Our goal will be to prove that in Theorem~\theoremmain, if we are only interested in monotonicity and stability for the multiplicity of the trivial representation in $H^i(C_n(M))$, it suffices to take $n\geq i+1$. We will do this by showing that the multiplicity of the trivial representation in $E_2^{p,q}(n)$ is stable for $n\geq p+q+1$ and monotone for $n\geq 1$. Indeed, we noted in Proposition~\ref{prop:monsingle} that monotonicity can be applied to the multiplicity $c_0(V_n)$ of the trivial representation $V(0)_n$ in $V_n$, so applying Proposition~\ref{prop:inductivestep} with $m=1$ and $\ell=1$ we conclude that $c_0(E_{r+1}^{p,q}(n))$ is constant for $n\geq p+q+1$ and monotone for $n\geq p+q$, which implies the theorem as in Step 4 of Theorem~\theoremmain.

Similarly, to prove Proposition~\ref{prop:improvebetti}, we will show that when $\dim M$ is odd, $c_0(E_2^{p,q}(n))$ is constant and monotone for $n\geq \big\lfloor\frac{p}{k}\big\rfloor$. In this case the spectral sequence degenerates immediately so the proposition follows. When $\dim M$ is even, we will prove that $c_0(E_2^{p,q(d-1)}(n))$ is constant for $n\geq \big\lfloor\frac{p}{k}\big\rfloor+q+1$ (we know it is monotone for all $n\geq 1$). Since $k\leq d-1$ by assumption, this quantity for $c_0(E_2^{p,q}(n))$ is bounded above by $\frac{p}{k}+\frac{q}{d-1}+1\leq\frac{p+q}{k}+1=\frac{1}{k}(p+q+k)$. We may thus apply Proposition~\ref{prop:inductivestep} with $m=\frac{1}{k}$ and $\ell=k$ to conclude that $c_0(E_{r+1}^{p,q}(n))$ is constant for $n\geq \frac{p+q}{k}+1$ and monotone for $n\geq \frac{p+q}{k}$, which implies that $c_0(H^i(B_n(M)))$ is constant and monotone for $n\geq \frac{i}{k}+1$ as desired.\\

To establish the desired stable range for $E_2^{p,q}(n)$, we analyze the summands $E(\mu)_n$; we will see that only certain components $E(\mu)_n$ can contribute a trivial subrepresentation to $E_2^{p,q}(n)$.
Recall that \[E(\mu)_n=\bigoplus_{\substack{\cS\text{ with}\\\overline{\cS}=\mu\langle n\rangle}} H^{q(d-1)}(C_\cS(\R^d))\otimes H^p(M^\cS).\]
Just as we previously wrote $E(\mu,r,\alpha)_n$ as an induced representation, we can choose a partition $\T$ of $\{1,\ldots,n\}$ with $\overline{\T}=\mu\langle n\rangle$ and write
\[E(\mu)_n=\Ind_{\Stab(\T)}^{S_n}H^{q(d-1)}(C_\T(\R^d))\otimes H^p(M^\T).\] Since $(\Ind_H^G V)^G\approx V^H,$ we only need to understand for which $\mu$ the $\Stab(\T)$--representation \[H(\T)\coloneq H^{q(d-1)}(C_\T(\R^d))\otimes H^p(M^\T)\] contains a trivial subrepresentation, and for what $n$ the corresponding $E(\mu)_n$ first appears.

If $\mu$ is nontrivial, meaning that for some entry $m$ of $\mu\langle n\rangle$ we have $m\geq 2$, let $T\subset \{1,\ldots,n\}$ be the corresponding element of $\T$ with $|T|=m$. By $S_m<S_n$ we mean the permutations which are the identity except on $T$. In most cases $S_m$, which is contained in $\Stab(\T)$, does not act trivially on $H(\T)$. By definition $M^\T$ is the subspace of $M^n$ constant on each element of $\T$, and in particular constant on those factors corresponding to $T$. It follows that $S_m$ acts trivially on $M^\T$, and thus also on its cohomology. The other factor is the top cohomology of $C_\T(\R^d)$, which splits as a product of smaller configuration spaces; one factor is $C_T(\R^d)=C_m(\R^d)$, and the other factors are fixed by $S_m$. Thus the top cohomology $H^{q(d-1)}(C_\T(\R^d))$ splits as a tensor product of $H^{(m-1)(d-1)}(C_m(\R^d))$ with factors on which $S_m$ acts trivially.

First, assume that $d$ is odd. In this case $H^{(m-1)(d-1)}(C_m(\R^d))$ is known to be isomorphic to $H^{2(m-1)}(C_m(\R^3))$, which never contains a trivial subrepresentation for any $m\geq 2$ \cite[Theorem 3.1]{Le}. This shows that $c_0(E(\mu)_n)=0$ unless $\mu=(0)$, in which case $q=0$. Thus when $\dim M$ is odd, only the row $(E_2^{p,0})^{S_n}$ is nonzero, so the spectral sequence degenerates immediately and $H^*(B_n(M))\approx (H^*(M)^{\otimes n})^{S_n}$. In particular, as $n\to \infty$, the cohomology $H^*(B_n(M))$ converges (degree by degree) to the free graded-commutative algebra $\Q[H^*(M)]=\Sym^* H^{\even}(M)\otimes \bwedge^* H^{\odd}(M)$. The Betti numbers of $B_n(M)$ for $\dim M$ odd were first computed by B\"odigheimer--Cohen--Taylor \cite{BCT}.

Corollary~\corollarymain\ and Proposition~\ref{prop:improvebetti} for $\dim M$ odd now follow immediately by applying the following elementary lemma to $V=H^*(M)$.
\begin{lemma}\label{lem:elementary}If $V$ is a graded vector space with $V^{(0)}=\Q$ and $V^{(1)}=\cdots=V^{(k-1)}=0$, then the degree--$p$ part of the invariants $(V^{\otimes n})^{S_n}$ stabilizes once $n\geq \big\lfloor\frac{p}{k}\big\rfloor$.
\end{lemma}
\begin{proof}Indeed for any graded vector space $V$ we have \[(V^{\otimes n})^{S_n}=\bigoplus_{a_0+a_1+a_2\cdots=n}\Sym^{a_0}V^{(0)}\otimes \Sym^{a_2}V^{(2)}\otimes \cdots\otimes \bwedge^{a_1}V^{(1)}\otimes \bwedge^{a_3}V^{(3)}\otimes \cdots\]
where the degree--$p$ part consists of those summands where $a_1+2a_2+3a_3+\cdots=p$. If $V^{(0)}=\Q$, by replacing $a_0$ with $a_0+1$ we can identify each such summand with an isomorphic summand of $(V^{\otimes n+1})^{S_{n+1}}$. Thus stabilization occurs once $n$ is large enough that $a_0$ is necessarily positive for all summands of $(V^{\otimes n+1})^{S_{n+1}}$. In general this happens once $n+1>p$; the last summand to appear is $\bwedge^p V^{(1)}$ with $a_1=p$. But the assumption that $V^{(1)}=\cdots=V^{(k-1)}=0$ means that any summand with $a_i>0$ for $1\leq i< k$ vanishes. Thus we only need to consider $ka_k+(k+1)a_{k+1}+\cdots=p$. Once $n+1>\frac{p}{k}$, it follows from $a_0+a_k+a_{k+1}+\cdots=n+1$ that $a_0>0$, so stabilization occurs once $n\geq \big\lfloor\frac{p}{k}\big\rfloor$ as claimed.
\end{proof}

Now consider the case that $d$ is even and $m\geq 3$. In this case it is known  (see e.g. Lehrer \cite[Equation 2.1]{Le}) that  $H^{(m-1)(d-1)}(C_m(\R^d))$ is isomorphic as an $S_m$--representation to $H^{m-1}(C_m(\R^2))$.
By Stanley \cite[Theorem 7.3]{St} or Lehrer--Solomon \cite[Theorem 3.9]{LS}, $H^{m-1}(C_m(\R^2))$ is isomorphic to $\Ind_{\langle(1\cdots m)\rangle}^{S_m}\epsilon\zeta$, where $\zeta$ is a faithful 1--dimensional representation and $\epsilon$ is the sign representation. Note that for $m\geq 3$, $\epsilon\zeta$ is nontrivial, so  $H^{m-1}(C_m(\R^d))$ contains no trivial subrepresentation.

It thus remains only to consider the case when $d$ is even and $\mu=(1,\ldots,1)\vdash q$.
Let $E(\mu,r)_n=\bigoplus_\alpha E(\mu,r,\alpha)_n$, and let $\T$ be the partition \[\T=\{\{1,q+1\},\ldots,\{q,2q\},\{2q+1\},\ldots,\{n\}\}.\]
 If $S_q<S_n$ is the subgroup inducing the same permutation on $\{1,\ldots,q\}$ as on ${\{q+1,\ldots,2q\}}$ and the identity elsewhere, we have $\Stab(\T)=S_2\wr S_q \times S_{n-2q}$. We have \[E(\mu,r)_n=\Ind^{S_n}_{\Stab(\T)}H(\T,r)\] where \[H(\T,r)\coloneq H^{q(d-1)}(C_\T(\R^d))\otimes H^{r}(M^{\T_{\geq 2}})\otimes H^{p-r}(M^{\T_1}).\] The subgroup $S_{n-2q}$ acts trivially on the first two factors, and the subgroup $S_2\wr S_q$ acts trivially on the last factor, so we can write the invariants as
\begin{equation}
\label{eq:HTr}
(H(\T,r))^{\Stab(\T)}= \big(H^{q(d-1)}(C_\T(\R^d))\otimes H^{r}(M^{\T_{\geq 2}})\big)^{S_2\wr S_q}\ \otimes\ \big(H^{p-r}(M^{\T_1})\big)^{S_{n-2q}}.
\end{equation}
We have $H^{q(d-1)}(C_\T(\R^d))\approx[H^{d-1}(C_2(\R^d))]^{\otimes q}$. Since $d$ is even, each $S_2$ factor acts trivially on this one-dimensional representation, while since $d-1$ is odd, the subgroup $S_q$ acts by the sign representation. Thus the $S_2\wr S_q$--invariants are the subspace of $H^{r}(M^{\T_{\geq 2}})=H^r(M^q)$ where each $S_2$ acts trivially and $S_q$ acts by the sign representation. Since each $S_2$ acts trivially all of on $H^*(M^{\T_{\geq 2}})$ the first condition is automatic, and we obtain: \begin{equation}\label{eq:Emurn}(E(\mu,r)_n)^{S_n}=(H^r(M^q))^\epsilon\otimes (H^{p-r}(M^{n-2q}))^{S_{n-2q}}\end{equation}
We point out that the first factor $(H^r(M^q))^\epsilon$ vanishes unless $r\geq kq-k$.  Indeed we have \[(H^*(M^q))^\epsilon=\bigoplus \bwedge^a H^{\even}(M)\otimes \Sym^b H^{\odd}(M).\]
The lowest degree term is either $\bwedge^1 H^0(M)\otimes \bwedge^{q-1}H^k(M)$ or $\bwedge^1 H^0(M)\otimes \Sym^{q-1}H^k(M)$, depending on whether $k$ is even or odd; in either case this term has degree $kq-k$.

By Lemma~\ref{lem:elementary}, the second factor of \eqref{eq:Emurn} stabilizes once $n-2q\geq \big\lfloor\frac{p-r}{k}\big\rfloor$.
Using our observation on the vanishing of $(H^r(M^q))^\epsilon$, for all nonzero summands $(E(\mu,r)_n)^{S_n}$ we have \[\big\lfloor\frac{p-r}{k}\big\rfloor\leq \big\lfloor\frac{p-kq+k}{k}\big\rfloor=\big\lfloor\frac{p}{k}\big\rfloor-q+1,\] so stabilization for $(E(\mu)_n)^{S_n}$ occurs once $n-2q\geq \big\lfloor\frac{p}{k}\big\rfloor-q+1$, or once $n\geq \big\lfloor\frac{p}{k}\big\rfloor+q+1$, as desired for Proposition~\ref{prop:improvebetti}. In particular, taking $k=1$ we obtain the desired stable range of $n\geq p+q+1$ for Corollary~\corollarymain.
We remark that the Betti numbers of $B_n(M)$ for $\dim M$ even have been determined by F\'{e}lix--Thomas \cite{FT1}.
\end{proof}

\begin{proof}[Proof of Theorem~\theoremvw]
We  deduce Theorem~\theoremvw\ from Theorem~\theoremmain.
Let $S_{n,\mu}$ denote the Young subgroup $S_{\mu_1}\times \cdots\times S_{\mu_k}\times S_{n-|\mu|}$. 
By the transfer isomorphism, the rational cohomology $H^i(B_{n,\mu}(M))$ is isomorphic to the $S_{n,\mu}$--invariants $(H^i(C_n(M)))^{S_{n,\mu}}$. By Theorem~\theoremmain, the multiplicity of the $S_n$--irreducible representation $V(\lambda)_n$ in $H^i(C_n(M))$ is constant for $n\geq 2i$ if $\dim M\geq 3$, or for $n\geq 4i$ if $\dim M=2$. We prove below in Lemma~\ref{lem:vw} that the multiplicity of the trivial $S_{n,\mu}$--representation in the $S_n$--irreducible representation $V(\lambda)_n$ is independent of $n$ once $n\geq 2|\mu|$. Thus once $n\geq \max(2i,2|\mu|)$ (or once $n\geq \max(4i,2|\mu|)$ if $\dim M=2$), the total multiplicity of the $S_{n,\mu}$--trivial representation in $H^i(C_n(M))$ is constant, as desired.
\end{proof}

\section{Monotonicity and induced representations}
\label{section:final}
The goal of this section is to prove the following theorem:
\begin{theorem:tech}
If $V$ is a fixed representation of $S_k$, the sequence of induced representations $\{\Ind_{S_k\times S_{n-k}}^{S_n} V\boxtimes \Q\}$ is monotone for $n\geq k$ and uniformly representation stable for $n\geq 2k$.
\end{theorem:tech}

There is another approach to Theorem~\ref{theorem:pieri}  by applying Schur--Weyl duality and working instead with representations of $\GL_n\Q$. This is carried out in Lemma 1.6 of the recent paper Sam--Weyman \cite{SW}.

\subsection{The branching rule}
\label{sec:branchingrule}
Given an irreducible $S_k$--representation $V_\lambda$, one can ask how the induced representation $\Ind_{S_k\times S_{n-k}}^{S_n} V_\lambda\boxtimes \Q$ decomposes into $S_n$--irreducibles. The answer is well-known, and is given by the \emph{branching rule}. The \emph{Young diagram} $Y_\lambda$ associated to a partition $\lambda\vdash k$ is an arrangement of $k$ boxes with $\lambda_1$ boxes in the first row, $\lambda_2$ boxes in the second row, and so on. Young diagrams serve as a very useful mental aid for many combinatorial constructions involved in the representation theory of the symmetric groups.

 For two partitions $\lambda\vdash k$ and $\mu\vdash n$, we write $\lambda\leadsto\mu$ if the Young diagram $Y_{\mu}$ is obtained from $Y_\lambda$ by adding $n-k$ boxes, with no two boxes added in the same column.
\begin{proposition}[Branching rule]
\label{prop:branching} Given a partition $\lambda\vdash k$, the induced representation $\Ind_{S_k\times S_{n-k}}^{S_n} V_\lambda\boxtimes \Q$ decomposes into irreducibles as
\begin{equation}
\label{eq:branching}
\Ind_{S_k\times S_{n-k}}^{S_n} V_\lambda\boxtimes \Q=\bigoplus_{\substack{\mu_i\vdash n\\\lambda\leadsto \mu_i}} V_{\mu_i}.
\end{equation}
\end{proposition}
(The name ``branching rule'' is sometimes reserved for the case when $n=k+1$, with the general case known as Pieri's rule. Both are special cases of the Littlewoood--Richardson rule.) We begin with an application of the branching rule.
\begin{lemma}
\label{lem:vw}
For $n\geq |\mu|$, let $S_{n,\mu}=S_{\mu_1}\times\cdots S_{\mu_k}\times S_{n-|\mu|}$. Then for fixed $\lambda$, the multiplicity of the trivial $S_{n,\mu}$--representation in the restriction of the irreducible $S_n$--representation $V(\lambda)_n$ is independent of $n$ for $n\geq 2|\mu|$.
\end{lemma}
\begin{proof}
Combining Frobenius reciprocity and repeated applications of the branching rule, the multiplicity in question is the number of sequences $(\nu^0,\nu^1,\ldots,\nu^k)$ of partitions satisfying:
\begin{itemize}
\item $\nu^0=(n-|\mu|)$, \item $|\nu^{i+1}|=|\nu^i|+\mu_{i+1}$ and $\nu^i\leadsto \nu^{i+1}$, and \item $\nu^k=\lambda[n]$.
\end{itemize} If we modify each $\nu^i$ by increasing its first entry by 1, we obtain a sequence $(\eta^0,\eta^1,\ldots,\eta^k)$ satisfying $\eta^0=(n+1-|\mu|)$, $|\eta^{i+1}|=|\eta^i|+\mu_{i+1}$ and $\eta^i\leadsto \eta^{i+1}$, and $\eta^k=\lambda[n+1]$. However, if $n\geq 2|\mu|$, the first row of each partition has length at least $n+1-|\mu|>\mu$, while the second row has length bounded by $\mu$. Thus once $n\geq 2|\mu|$, given any such sequence $(\eta^0,\eta^1,\ldots,\eta^k)$ we can decrease the first entry of each $\eta^i$ by 1 to recover a valid sequence $(\nu^0,\nu^1,\ldots,\nu^k)$. This bijection demonstrates that the multiplicity in question is constant for $n\geq 2|\mu|$, as desired.
\end{proof}

The branching rule and Pieri's rule can be found in any standard reference; the usual proofs are by character theory \cite[A.7]{FH}, \cite[Corollary 7.3.3]{Fu}  or using Frobenius reciprocity \cite[Theorem 2.8.3]{Sa}. These proofs show that the  irreducibles $V_{\mu_i}$ with $\lambda\leadsto \mu_i$ occur with multiplicity 1.
However, we will need an explicit description of the decomposition in \eqref{eq:branching}. We were not able to find this decomposition in the literature, so we describe it here. This requires first reviewing the explicit construction of irreducible $S_n$--representations by polytabloids.

\para{Tableaux, tabloids, and polytabloids}
Fix a partition $\lambda\vdash n$. A Young tableau, or simply tableau, of shape $\lambda$ is a bijective labeling of the boxes of the Young diagram $Y_\lambda$ by the numbers $\{1,\ldots,n\}$. A \emph{Young tabloid}, or simply \emph{tabloid}, of shape $\lambda$ is an equivalence class of tableaux under reordering of the numbers within each row. If $T$ is a tableau of shape $\lambda$, we denote by $\{T\}$ the corresponding tabloid. We denote by $M^\lambda$ the $\Q$-vector space with basis the tabloids of shape $\lambda$.

The irreducible representation $V_\lambda$ naturally sits inside $M^\lambda$, as follows. We associate to each tableau $T$ a certain element $v_T\in M^\lambda$ called a \emph{polytabloid}, and the span of the elements $v_T$ over all tableaux $T$ of shape $\lambda$ is the irreducible representation $V_\lambda$. The polytabloid $v_T$ is the sum of the tabloids of all tableaux obtained from $T$ by reordering within each column, with sign determined by the sign of the reordering. Formally, let $\ColStab(T)<S_n$ be the subgroup preserving the set of entries of each column of $T$. We write $qT$ for the action of $S_n$ on tableaux, and $\{qT\}$ for the associated tabloid. Then we have:
\begin{equation}
\label{eq:vT}
v_T\coloneq \sum_{q\in\ColStab(T)}(-1)^q \{qT\}
\end{equation}
 Note that if $T$ and $T'$ differ only by reordering within columns, the sum in the definition is over the same collection of tableaux and we have $v_T=\pm v_{T'}$.

When dealing later with induced representations, we will need the following definition. Given a partition $\lambda\vdash k$, define an \emph{$n$--pseudo-tableau $T$ of shape $\lambda$} to be a bijective labeling of the boxes of $Y_\lambda$ by some $k$--element subset $\supp(T)\subset\{1,\ldots,n\}$. Given such a pseudo-tableau, the associated pseudo-tabloid $\{T\}$ is the equivalence class of $T$ under reorderings of the numbers within each row.

\subsection{Decomposing induced representations}
Fix a partition $\lambda\vdash k$ with corresponding irreducible representation $V_\lambda$. Our goal in this section is to describe, for each irreducible representation $V_\mu$ contained in the $S_n$--representation $\Ind_{S_k\times S_{n-k}}^{S_n}V_\lambda\boxtimes \Q$, an explicit generator for a subrepresentation $W^\mu$ which contains $V_\mu$.

Since $V_\lambda$ is contained in $M^\lambda$, the induced representation $I_n(V_\lambda)=\Ind_{S_k\times S_{n-k}}^{S_n}V_\lambda\boxtimes \Q$ is naturally contained in $I_n(M^\lambda)=\Ind_{S_k\times S_{n-k}}^{S_n}M^\lambda\boxtimes \Q$. Since $M^\lambda$ has basis the tabloids of shape $\lambda$, the induced representation $I_n(M^\lambda)$ naturally has basis the $n$--pseudo-tabloids of shape $\lambda$. (This is easily checked using Criterion~\ref{crit:ind}.) Given an $n$--pseudo-tableau $T$, the associated polytabloid $v_T\in I_n(M^\lambda)$ is defined by the same formula \eqref{eq:vT}. By Criterion~\ref{crit:ind},  $I_n(V_\lambda)$ is exactly the subspace of $I_n(M^\lambda)$ spanned by $v_T$ for all $n$--pseudo-tableaux $T$.

\para{Defining the maps $\pi_\mu$} For each partition $\mu\vdash n$ with $\lambda\leadsto \mu$, there is a natural map \[\pi_\mu\colon I_n(M^\lambda)\to M^\mu\] defined on basis elements as follows. Informally, an $n$--pseudo-tableau of shape $\lambda$ is labeled by $k$ elements of $\{1,\ldots,n\}$, and the diagram $Y_\mu$ differs from $Y_\lambda$ by $n-k$ boxes. The map $\pi_\mu$ is defined by filling those $n-k$ boxes with the remaining $n-k$ elements of $\{1,\ldots,n\}$, and summing the resulting tableaux over all ways of doing do. Formally, let $\B=\{b_1,\ldots,b_{n-k}\}$ be the boxes present in $Y_\mu$ but not $Y_\lambda$.
Given an $n$--pseudo-tableau $T$ of shape $\lambda$, let $Y$ be the complement $\{1,\ldots,n\}\setminus \supp(T)$; we have $|Y|=n-k$. Let $\Bij(\B,Y)$ be the set of bijections $f\colon \B\to Y$. For $f\in \Bij(\B,Y)$, we define $T_f$ to be the tableau of shape $\mu$ obtained from $T$ by filling the box $b_i$ with the label $f(b_i)$. Define the map $\pi_\mu$ on tableaux by: \[\pi_\mu(T)=\sum_{f\in \Bij(\B,Y)}T_f\] Slightly abusing notation, we use the same notation for the induced map on tabloids $\pi_\mu\colon I_n(M^\lambda)\to M^\mu$ defined by $\pi_\mu(\{T\})=\{\pi_\mu(T)\}$. It is clear from the definition that the map $\pi_\mu$ respects the action of $S_n$.

It is easiest to illustrate this when $n=k+1$, since there is only one box to fill. For example, if $\lambda=\tiny\yng(3,2,1)$, in this case we might have: \[\mu_1=\yng(4,2,1)\quad\mu_2=\yng(3,3,1)\quad\mu_3=\yng(3,2,2)\quad\mu_4=\yng(3,2,1,1)\]
Taking for example the pseudo-tableau $T=\tiny\young(721,53,4)$,
 we have:
\[\pi_{\mu_1}(T)=\young(7216,53,4)\quad\ \pi_{\mu_2}(T)=\young(721,536,4)\quad\ \pi_{\mu_3}(T)=\young(721,53,46)\quad\ \pi_{\mu_4}(T)=\young(721,53,4,6)\]

\noindent\textbf{Notational remark}: In the remainder of the paper, we will always use $\T$ to denote a tableau of shape $\mu$ for some $\mu\vdash n$ with $\lambda\leadsto \mu$. 
We will use $T$ to denote an $n$--pseudo-tableau
of shape $\lambda$.

\para{Defining the subspaces $W^\mu$} If $\T$ is a tableau of shape $\mu$,  deleting the boxes $b_1,\ldots,b_{n-k}$ together with their labels yields an $n$--pseudo-tableau $\widetilde{\T}$ of shape $\lambda$. Note that this operation does not descend to tabloids. We associate to each tableau $\T$ of shape $\mu$ an element $w_\T$ of $I_n(M^\lambda)$ as follows:
\[w_\T\coloneq \sum_{q\in\ColStab(\T)}(-1)^q \{\widetilde{q\T}\}\]

As before, if $\T$ and $\T'$ differ by reordering within columns, we have $w_\T=\pm w_{\T'}$. We will see below that $w_\T$ is contained in $I_n(V_\lambda)=\spn \{v_T\}$. In the following examples, we have written $w_\T$ as a sum of polytabloids $v_T$ to illustrate this containment:
\begin{align*}
\T=\young(1234,56,7)&\qquad\qquad  w_\T
=v_{\tiny\young(123,56,7)}\\
\T=\young(123,456,7)\quad&\qquad\qquad  w_\T
=v_{\tiny\young(123,45,7)}-v_{\tiny\young(126,45,7)}\\
\T=\young(123,45,67)\quad&\qquad\qquad w_\T=v_{\tiny\young(123,45,6)}+v_{\tiny\young(173,42,6)}+v_{\tiny\young(153,47,6)}\\
\T=\young(123,45,6,7)\quad&\qquad\qquad w_\T=v_{\tiny\young(123,45,6)}-v_{\tiny\young(723,15,4)}+v_{\tiny\young(623,75,1)}-v_{\tiny\young(423,65,7)}
\end{align*}\pagebreak

\begin{definition}
We define $W^\mu<I_n(M^\lambda)$ to be the subspace spanned by the $w_{\T}$ over all tableaux $\T$ of shape $\mu$ (or equivalently, the $S_n$--span of any such $w_{\T}$, since any tableau is obtained from any other by a permutation).
\end{definition}

Let $\preceq$ be the lexicographic order on partitions of $n$, so that the partition $(n)$ is \emph{larger} than any other. The properties of $W^\mu$ that we will need are as follows.

\begin{proposition}
\label{prop:filtration}
The subspace $W^\mu$ is contained in $I_n(V_\lambda)$, and in fact is contained in $\displaystyle{\bigoplus _{\nu\preceq \mu}V_\nu}$. The map $\pi_\mu\colon I_n(V_\lambda)\to M^\mu$ restricts to a surjection $\pi_\mu\colon W^\mu\twoheadrightarrow V_\mu$. 
\end{proposition}
\begin{proof} Fix $\T$ a tableau of shape $\mu$. The proposition will follow from these claims:
\begin{enumerate}[\qquad\text{Claim }1.]
\item $w_\T\in I_n(V_\lambda)$.
\item $\pi_\mu(w_{\T})=c\cdot v_\T$ for a fixed nonzero constant $c\in \Z$.
\item $\pi_\nu(w_{\T})=0$ for any $\nu\succ\mu$ with $\lambda\leadsto\nu$.
\end{enumerate}
Since $W^\mu$ is spanned by $w_\T$, Claim 1 implies that $W^\mu$ is contained in $I_n(V_\lambda)$. Since $V_\mu$ is spanned by $v_\T$, Claim 2 implies that the restriction of $\pi_\mu$ to $W^\mu$ surjects onto $V_\mu$. By Proposition~\ref{prop:branching}, each irreducible $V_\nu$ with $\nu\vdash n$ and $\lambda\leadsto \nu$ occurs in $I_n(V_\lambda)$ with multiplicity 1. Thus to verify that $\displaystyle{W^\mu<\bigoplus_{\nu\preceq \mu}V_\nu}$, it suffices to show that $V_\nu\not<W^\mu$ for all $\nu\succ \mu$. But if $\nu\succ \mu$, consider the map $\pi_\nu\colon I_n(V_\lambda)\to M^\nu$. By Claim 2, the image of $\pi_\nu$ contains $V_\nu$, but Claim 3 implies that $W^\mu$ is contained in the kernel of $\pi_\nu$. It follows that $V_\nu\not<W^\mu$ for all $\nu\succ \mu$, and so $\displaystyle{W^\mu<\bigoplus_{\nu\preceq \mu}V_\nu}$ as desired.

\para{Claim 1} We must show that $w_\T$ can be written as a linear combination of the polytabloids $v_T$. The examples above are designed to illustrate this argument. In fact, we will prove that
\begin{equation}
\label{eq:colstab}
c\cdot w_\T=\sum_{s\in\ColStab(\T)}(-1)^s v_{\widetilde{s\T}}
\end{equation} for an explicit nonzero constant $c\in \Z$. We remark that this sum is redundant: if $s,s'\in\ColStab(\T)$ are such that $\widetilde{s\T}$ is obtained from $\widetilde{s'\T}$ by reordering within columns, we have already noted that $(-1)^sv_{\widetilde{s\T}}=(-1)^{s'} v_{\widetilde{s'\T}}$. Thus we could eliminate the constant $c$ (proving that $w_\T$ is in the $\Z$--span of the $\{v_T\}$, not just the $\Q$--span) by restricting the sum to coset representatives for $\ColStab(\widetilde{\T})$ inside $\ColStab(\T)$. A similar remark applies to the constant in Claim 2.

Note that $q\widetilde{s\T}=\widetilde{qs\T}$ for $q\in\ColStab(\widetilde{s\T})$, so
the right side of \eqref{eq:colstab} is equal to \begin{align*}
\sum_{s\in\ColStab(\T)}(-1)^s v_{\widetilde{s\T}}
&=\sum_{s\in \ColStab(\T)}\sum_{q\in\ColStab(\widetilde{s\T})}(-1)^s(-1)^q \{q\widetilde{s\T}\}\\
&=\sum_{s\in \ColStab(\T)}\sum_{q\in\ColStab(\widetilde{s\T})}(-1)^{qs} \{\widetilde{qs\T}\}.
\end{align*}

Since $\ColStab(\widetilde{s\T})<\ColStab(\T)$ for each $s\in \ColStab(\T)$, we can collect terms and write this as a sum over $p=qs$ of $(-1)^p\{\widetilde{p\T}\}$. It remains to check that each such term occurs the same number of times. Note that the stabilizers $\ColStab(\widetilde{s\T})$ for different $s$ are all conjugate. The constant $c$ will be their cardinality: $c\coloneq \big\vert \ColStab(\widetilde{\T})\big\vert$. Also note that if $p=qs$ and $q\in\ColStab(\widetilde{s\T})$ we have $\ColStab(\widetilde{s\T})=\ColStab(\widetilde{p\T})$. Thus we may rewrite the sum above as
\begin{align*}\sum_{s\in\ColStab(\T)}(-1)^s v_{\widetilde{s\T}}&=\sum_{p\in\ColStab(\T)}\sum_{\substack{s\in \ColStab(\T)\\q\in\ColStab(\widetilde{p\T})\\p=qs}}(-1)^p \{\widetilde{p\T}\}\\&=\big\vert \ColStab(\widetilde{\T})\big\vert\sum_{p\in\ColStab(\T)}(-1)^p\{\widetilde{p\T}\}\\&=c\cdot w_\T,\end{align*} as desired.

\para{Claims 2 and 3} The proofs of the remaining two claims are essentially the same, so we combine them. We remark that in the previous claim we were essentially working with tableaux, and the argument did not depend on passing to equivalence classes of tableaux (that is, to tabloids). In the proof of these claims, however, it will be essential that we deal with tabloids.

Fix $\nu\vdash n$ with $\lambda\leadsto \nu$ and $\nu\succeq \mu$. Let $\T$ be a tableau of shape $\mu$, and consider $\pi_\nu(w_\T)$. Our goal is to show that if $\nu\succ\mu$ then $\pi_\nu(w_\T)=0$, while if $\nu=\mu$ we have $\pi_\mu(w_\T)=v_\T$.

If $\T$ is a tableau of shape $\mu$ and $g\in \Bij(\B_\mu,\B_\nu)$, define $\T_g$ to be the tableau of shape $\nu$ obtained from $\T$ by moving the entry in box $b^\mu_i$ to the box $g(b^\mu_i)$.
The key claim is that \begin{equation}
\label{eq:piwT}
\pi_\nu(w_\T)=\sum_{\substack{q\in\ColStab(\T)\\g\in\Bij(\B_\mu,\B_\nu)}}(-1)^q\big\{(q\T)_g\big\}.
\end{equation} To see this, recall that passing from $\T$ to $w_\T$ amounts to deleting the contents of the boxes $\B_\mu$ (summing over all permutions $q$ within the columns); conversely, passing from $T$ to $\pi_\nu(T)$ amounts to filling in the boxes $B_\nu$ (summing over all ways $f$ of filling them in). Grouped together, the process of deleting the contents of the boxes $B_\mu$, then placing them in the boxes $B_\nu$, is equivalent to just moving the entries in the boxes $B_\mu$ to the boxes $B_\nu$ according to some identification $g$. Since no two boxes from $B_\mu$ or from $B_\nu$ ever occur in the same column (since $\lambda\leadsto \mu$ and $\lambda\leadsto \nu$), the permutation $q\in\ColStab(\T)$ is independent of the bijection $g\in \Bij(B_\mu,B_\nu)$, giving the desired equation \eqref{eq:piwT}.

We say that a bijection $g\in \Bij(\B_\mu,B_\nu)$ is \emph{good} if for each box $b^\mu_i$, its image $g(b^\mu_i)$ lies in the same row as $b^\mu_i$; otherwise we call the bijection $g$ \emph{bad}. Note that \emph{all} bijections are bad if $\nu\neq\mu$, and when $\nu=\mu$ a permutation $g$ is good if and only if the tabloids $\{\T\}$ and $\{\T_g\}$ are equivalent (since $g$ only permutes elements within rows).

The key remaining step is the following, which is carried out below: for each bad bijection $g$ we will divide up the permutations $q\in\ColStab(\T)$ into pairs $q\corr_g q'$ so that $q$ and $q'$ have opposite sign, but $\big\{(q\T)_g\big\}=\big\{(q'\T)_g\big\}$. This implies that the corresponding terms in \eqref{eq:piwT} corresponding to $(q,g)$ and $(q',g)$ will cancel: \[(-1)^q\big\{(q\T)_g\big\}+(-1)^{q'}\big\{(q'\T)_{g'}\big\}=\pm \big(\big\{(q\T)_g\big\}-\big\{(q'\T)_{g'}\big\}\big)=0.\] Once we have constructed such a pairing, we can conclude that all the terms in \eqref{eq:piwT} corresponding to bad bijections will disappear.
If $\nu\neq \mu$, all bijections are bad, so the entire sum in \eqref{eq:piwT} cancels, showing $\pi_\nu(w_\T)=0$ as desired. If $\nu=\mu$, only the good bijections remain; let $c\geq 1$ be the number of good bijections in $\Bij(\B_\mu,\B_\mu)$. If $g$ is good we have $\{(q\T)_g\}=\{q\T\}$, so we conclude that \[\pi_\mu(w_\T)=\sum_{\substack{q\in\ColStab(\T)\\g\text{ good}}}(-1)^q\{(q\T)_g\}=\sum_{q\in\ColStab(\T)}c(-1)^q\{q\T\}=c\cdot v_\T\]
as desired.

It remains only to define the pairing $q\corr_g q'$. Fix a bad bijection $g$, and let $B\in \B_\mu$ be the rightmost box for which $g(b)$ is \emph{higher} than $b$. At least one box changes height (since $g$ is bad), and not all boxes can move lower (since $\nu\succeq \mu$), so some such box exists. Let $C$ be the box which lies in the same column as $B$ and the same row as $g(B)$ (here we use that $g(B)$ is higher than $B$). Note that since $C$ lies in the same column as $B$, we know that $C\not\in \B_\mu$, so $C$ is not moved by the operation $\T\mapsto \T_g$.

Given $q$, the corresponding $q'$ is uniquely defined by the property that $q\T$ and $q'\T$ agree, except that the labels in box $B$ and box $C$ are exchanged. It follows that $(q\T)_g$ and $(q'\T)_g$ agree, except that the labels in box $g(B)$ and box $C$ are exchanged. However, the boxes $g(B)$ and  $C$ lie in the same row, so the associated tabloids $\{(q\T)_g\}$ and $\{(q'\T)_g\}$ are equivalent. Note that $q$ and $q'$ differ by a transposition, so they have opposite signs as desired. Finally, note that $q$ also corresponds to $q'$, so this rule does divide up the permutations in $\ColStab(\T)$ into pairs $q\corr_g q'$. This concludes the proof of Claims 2 and 3, and thus of Proposition~\ref{prop:filtration}.
\end{proof}

For example, if $\mu=\mu_4$ and $\nu=\mu_2$ as in the previous example, we would have:
\begin{align*}q\T&=\young(123,45,6,7)=\young(\cell\cell\cell,\capC\cell,\cell,\capB)\qquad\corr\qquad q'\T=\young(123,75,6,4)=\young(\cell\cell\cell,\capB\cell,\cell,\capC)\\[4pt]
(q&\T)_g=\young(123,457,6)\qquad\qquad\qquad\qquad\quad\ \ \,
(q'\T)_g=\young(123,754,6)
\end{align*}
and these tableaux define the same tabloid: {\tiny$\left\{\young(123,457,6)\right\}=\left\{\young(123,754,6)\right\}$}

\subsection{Monotonicity}
\label{section:endmono}
Recall that if $V$ is an $S_k$--representation, $I_n(V)$ is the $S_n$--representation $\Ind_{S_k\times S_{n-k}}^{S_n}V\boxtimes \Q$. Our goal in this section is to prove Theorem~\ref{theorem:pieri}. By Proposition~\ref{prop:monsurj}, we need only to prove that the sequence $\{I_n(V)\}$ is monotone with respect to the inclusions $\iota_n\colon I_n(V)\to I_{n+1}(V)$.
For any partition $\mu\vdash n$, define $\mu\{n+1\}=(\mu_1+1,\mu_2,\ldots,\mu_\ell)\vdash n+1$. This is just a reindexed version of the notation $\lambda[n]$ used throughout the paper; in particular, $\lambda[n]\{n+1\}=\lambda[n+1]$.

\begin{proof}[Proof of Theorem~\ref{theorem:pieri}.]
Since every representation $V$ is the direct sum of irreducible representations, we may assume that $V$ is irreducible. Our goal is to prove that for fixed $\lambda\vdash k$, the sequence $\{I_n(V_\lambda)\}$ is monotone with respect to the inclusions $I_n(V_\lambda)\hookrightarrow I_{n+1}(V_\lambda)$. By Proposition~\ref{prop:branching}, each irreducible $V_\mu$ with $\mu\vdash n$ and $\lambda\leadsto \mu$ occurs in $I_n(V_\lambda)$ with multiplicity 1. Thus it suffices to show that $S_{n+1}\cdot V_\mu$ contains $V_{\mu\{n+1\}}$. 

Unfortunately, it is quite difficult to identify $V_\mu$ as a subrepresentation of $I_n(V_\lambda)$. But we know that $V_\mu$ is contained in the subrepresentation $W^\mu$, which was explicitly described in the previous section. We first show that $S_{n+1}\cdot W^\mu$ surjects to $V_{\mu\{n+1\}}$.
To do this, fix a tableau $\T$ of shape $\mu$, and let $\T_{\{n+1\}}$ be the tableau of shape $\mu\{n+1\}$ obtained from $\T$ by filling the additional box with $n+1$. Recall that $W^\mu$ is generated by $w_\T$. The key claim is that $\iota_n(w_\T)=w_{\T_{\{n+1\}}}$.

We saw above that $I_n(M^\lambda)$ has basis the $n$--pseudo-tabloids of shape $\lambda$. That is, a basis element is a subset $S\subset \{1,\ldots,n\}$ with $|S|=k$ together with a labeling $T$ of $Y_\lambda$ by the elements of $S$. Then by Criterion~\ref{crit:ind}, a basis for $\Ind_{S_n}^{S_{n+1}}I_n(M^\lambda)$ is given by a subset $S\subset \{1,\ldots,n+1\}$ with $|S|=k$, a labeling $T$ of $Y_\lambda$ by the elements of $S$, and an element $x\in \{1,\ldots,n+1\}$ with $x\not\in S$. The natural inclusion $I_n(M^\lambda)\hookrightarrow \Ind_{S_n}^{S_{n+1}}I_n(M^\lambda)$ takes the pair $(S,T)$ to the triple $(S,T,n+1)$. The projection to $I_{n+1}(M^\lambda)$ is given by forgetting the element $x$, so the composition takes a pair $(S\subset\{1,\ldots,n\},T)$ to the same pair $(S\subset \{1,\ldots,n+1\},T)$. Thus the map $\iota_n\colon I_n(M^\lambda)\to I_{n+1}(M^\lambda)$ simply takes an $n$--pseudo-tabloid of shape $\lambda$ and considers it as an $(n+1)$--pseudo-tabloid.

But the $n$--pseudo-tableau $\widetilde{\T}$ coincides with $\widetilde{\T_{\{n+1\}}}$ when considered as an $(n+1)$--pseudo-tableau, since $\T$ and $\T_{\{n+1\}}$ are identical except for one box, which is deleted when passing from $\T_{\{n+1\}}$ to $\widetilde{\T_{\{n+1\}}}$. Since $\ColStab(\T)=\ColStab(\T_{\{n+1\}})$, the sums defining $w_\T\in I_n(V_\lambda)$ and $w_{\T_{\{n+1\}}}\in I_{n+1}(V_\lambda)$ coincide. This verifies the claim that
 $\iota_n(w_\T)=w_{\T_{\{n+1\}}}$.
We proved in Proposition~\ref{prop:filtration} that $\pi_{\mu\{n+1\}}(w_{\T_{\{n+1\}}})=c\cdot v_{\T_{\{n+1\}}}$. Since $v_{\T_{\{n+1\}}}$ generates $V_{\mu\{n+1\}}$, this shows that $S_{n+1}\cdot W^\mu$  surjects to $V_{\mu\{n+1\}}$ as desired.\\

Of course, $W^\mu$ also contains other irreducible representations $V_\nu$ with $\nu\prec \mu$.\footnote{With a bit of work, we can show that $W^\nu$ is contained in $W^\mu$ if $\nu\prec \mu$ and no box $b^\mu_i\in\B_\mu$ appears to the left of the corresponding box $b^\nu_i\in\B_\nu$, numbering boxes from left to right. I conjecture that $W^\mu$ is the sum of $V_\nu$ for exactly these $\nu$.} However, by the definition \eqref{eq:iotan} of $\iota_n$, the image $S_{n+1}\cdot V_\nu$ factors through $V_\nu\hookrightarrow \Ind_{S_n}^{S_{n+1}}V_\nu$. By the branching rule (Proposition~\ref{prop:branching}), $\Ind_{S_n}^{S_{n+1}}V_\nu$ decomposes as $\bigoplus V_\eta$ over those $\eta\vdash n+1$ with $\nu\leadsto \eta$. It is easy to see that $\nu\{n+1\}$ is lexicographically larger than any other $\eta\vdash n+1$ with $\nu\leadsto \eta$, since $\nu\{n+1\}_1=\nu_1+1$, while $\eta_1=\nu_1$ for any other such $\eta$.
Thus the image $S_{n+1}\cdot V_\nu$ only contains irreducibles $V_\eta$ with $\eta\preceq \nu\{n+1\}$.

Since $\nu\prec\mu$ implies $\nu\{n+1\}\prec \mu\{n+1\}$, we conclude that $S_{n+1}\cdot V_\nu$ contains only irreducibles $V_\eta$ with $\eta\prec \mu\{n+1\}$. It follows that the surjection $S_{n+1}\cdot W^\mu\to V_{\mu\{n+1\}}$ vanishes on the image of each irreducible $V_\nu<W^\mu$ except for $V_\mu$ itself. Thus this surjection restricts to a surjection $S_{n+1}\cdot V_\mu\twoheadrightarrow V_{\mu\{n+1\}}$, verifying that $\{I_n(V_\lambda)\}$ is monotone, as desired.
\end{proof}

\small

\noindent
Dept. of Mathematics\\
University of Chicago\\
5734 University Ave.\\
Chicago, IL 60637\\
E-mail: tchurch@math.uchicago.edu
\medskip

\end{document}